\documentclass[12pt, leqno]{amsart}
\usepackage{amsmath}
\usepackage{amssymb}
\usepackage{amsthm}
\usepackage{enumerate}
\usepackage[mathscr]{eucal}
\usepackage{xcolor}
\theoremstyle{plain}
\usepackage{tikz}
\usepackage{soul}
\usepackage[normalem]{ulem}
\newtheorem{theorem}{Theorem}[section]
\newtheorem{lemma}[theorem]{Lemma}
\newtheorem{prop}[theorem]{Proposition}
\theoremstyle{definition}
\newtheorem{definition}[theorem]{Definition}
\newtheorem{remark}[theorem]{Remark}

\newtheorem{cor}[theorem]{Corollary}
\theoremstyle{remark}

\begin{document}
	\title[ On symmetric functions and symmetric operators on Banach spaces]{ On symmetricity of orthogonality in function spaces and space of operators on Banach spaces}
	
\author[Sohel, Sain and  Paul]{Shamim Sohel, Debmalya Sain and  Kallol Paul }

\address[Sohel]{Department of Mathematics, Jadavpur University, Kolkata 700032, India.}	\email{shamimsohel11@gmail.com}

\address[Sain]{Department of Mathematics\\ Indian Institute of Information Technology, Raichur\\ Karnataka 584135 \\INDIA}	\email{saindebmalya@gmail.com}

\address[Paul]{Vice-Chancellor, Kalyani University,  West Bengal 741235, India \&
	Department of Mathematics, Jadavpur University (on lien), Kolkata 700032, India.}	\email{kalloldada@gmail.com}

\thanks{The first  author would like to thank  CSIR, Govt. of India, for the financial support in the form of Senior Research Fellowship under the mentorship of Prof. Kallol Paul. }
	

	\begin{abstract}
	We study symmetric points with respect to $(\rho_+)$-orthogonality, $(\rho_{-})$-orthogonality and $\rho$-orthogonality in the space $C(K, \mathbb{X}),$ where $K$ is  a perfectly normal, compact space and $ \mathbb X$ is a Banach space. We characterize left symmetric points and right symmetric points in $C(K, \mathbb{X})$ with respect to $(\rho_{+})$-orthogonality and $(\rho_{-})$-orthogonality, separately. Furthermore, we provide necessary conditions for left symmetric and right symmetric points with respect to $\rho$-orthogonality. As an application of these results we also study these symmetric points in the space of  operators  defined on some special Banach spaces.
	\end{abstract}
	

	\maketitle
	\section{Introduction}

	In the context of Banach space geometry, the concept of symmetric points with respect to Birkhoff-James orthogonality plays a significant role and has been extensively studied in recent years. These investigations, motivated by their applications in the isometric theory of Banach spaces (see\cite{CSS,GSP, PMW, PSS, S1, SRBB, T}), provide valuable insights into the geometric properties of such spaces, including the analysis of onto isometries.
	In this article, we extend this line of inquiry by examining symmetric points with respect to $\rho$-orthogonality, $(\rho_{+})$-orthogonality and $(\rho_{-})$-orthogonality in the space $C(K, \mathbb{X})$ as well as the space of operators defined on some specific Banach spaces.

	We use the symbols $ \mathbb{X}, \mathbb{Y}$ to denote real Banach spaces. Let $ B_{\mathbb{X}}= \{ x \in \mathbb{X}: \|x\| \leq 1\} $  and $ S_{\mathbb{X}}= \{ x \in \mathbb{X}: \|x\| = 1\} $ denote  the unit ball and the unit sphere of $\mathbb{X},$ respectively. The  dual space of $ \mathbb{X}$ is denoted by  $ \mathbb{X}^*.$  Let $ \mathbb{L}(\mathbb{X}, \mathbb{Y}) $ $(\mathbb{K}(\mathbb{X}, \mathbb{Y}) )$ denote the Banach space of all bounded (compact) linear operators  from $\mathbb{X}$ to $\mathbb{Y}.$ The convex hull of a non-empty set $ S  \subset \mathbb{X}$  is the intersection of all convex sets in $\mathbb{X}$ containing $S$ and it is denoted by $co (S).$  For a non-empty convex set $A,$ an element $z \in A$ is said to be an extreme point of $ A $ if the equality $z = (1 - t)x + t y$, with $t \in (0, 1)$ and $x, y \in A,$ implies that $ x = y = z.$ The set of all extreme points of $A$ is denoted by $Ext(A).$ A Banach space is said to be strictly convex if every point of $S_{\mathbb{X}}$ is an extreme point of $B_{\mathbb{X}}.$
	For  $T \in \mathbb{L}(\mathbb{X}, \mathbb{Y}),$ the norm attainment set of $T,$  denoted by $M_T,$ is defined as $M_T=\{x \in S_{\mathbb{X}}: \|Tx\|=\|T\|\}.$ 	Given any non-zero $ x \in \mathbb{X},$ $ f \in S_{\mathbb{X}^*} $ is said to be a support functional at $x$  if $ f(x)= \|x\|.$ Let $ J(x) := \{ f \in S_{\mathbb{X}^*}: f(x)= \|x\|\}.$  Given a subset $M$ of $\mathbb{X}^*,$ $\overline{M}^{w^*}$ denotes the closure of $M$ with respect to the weak*- topology defined on $\mathbb{X}^*.$ For a Banach space $\mathbb{X}$ and $x,y \in \mathbb{X},$ $V(\mathbb{X}, x,y)= co(\{ x^*(y): x^* \in S_\mathbb{X}, x^*(x)=1\}).$

	Let us now recall from  \cite{B, J} that  an element $x \in \mathbb{X}$ is said to be Birkhoff-James orthogonal to $y \in \mathbb{X}$ if $ \| x + \lambda y \| \geq \|x\|,$ for all $\lambda \in \mathbb{K}.$   Symbolically, it is written as $ x \perp_B y.$ 
	       	It is clear that in general, Birkhoff-James orthogonality is  not symmetric. In this context the notions of left and right symmetric  points were introduced in \cite{S1}:
	
	\begin{definition}
		An element $x \in \mathbb{X}$ is said to be a left symmetric point if $x \perp_B y$ implies
		that $y \perp_B x$ for all $y \in \mathbb{X}.$ Similarly, a point $ x \in \mathbb{X}$ is said to be a right symmetric point if
		$y \perp_B x$ implies that $x \perp_B y$ for all $y \in  \mathbb{X}.$ An element $x \in \mathbb{X}$ is said to be a symmetric point if $x$ is both left symmetric and right symmetric.   
	\end{definition}

		Let us now mention the  definition of $\rho$-orthogonality studied in \cite{CW,Mi}.  
	
	\begin{definition}
		Let $\mathbb{X}$ be a normed linear space and let $x, y \in \mathbb{X}.$ The norm derivatives at $x$ in the direction of $y$ is defined as:
		\begin{eqnarray*}
			\rho'_{+}(x, y)&=& \|x\| \lim_{t \to 0^+} \frac{\|x+ty\|-\|x\|}{t} \\
			\rho'_{-}(x, y)&=& \|x\| \lim_{t \to 0^-} \frac{\|x+ty\|-\|x\|}{t}\\
			\rho'(x, y) &=& \frac{1}{2}(\rho'_+(x, y) + \rho'_-(x, y)).
   		\end{eqnarray*}
		We say that $x$ is $\rho_+$-orthogonal to $y$, i.e., $x \perp_{\rho_+} y$ if $\rho'_+(x, y)=0$ and  $x$ is $\rho_-$-orthogonal to $y$, i.e., $x \perp_{\rho_-} y$ if $\rho'_-(x, y)=0.$
		Moreover,	$x$ is $\rho$-orthogonal to $y$, i.e., $x \perp_{\rho} y$ if $\rho'(x, y)=0.$ Note that $\rho_+$-orthogonality, $ \rho_-$-orthogonality and $\rho$-orthogonality are homogeneous.
	\end{definition}
	
	Next we observe some of the important results regarding the functions $\rho'_+$ and $\rho'_-.$ 
	
	\begin{lemma}\label{functional}\cite[Th. 2.3, 2.4]{Woj}
		Let $\mathbb{X}$ be a normed linear space. Then for $x, y \in S_\mathbb{X},$ 
		\begin{eqnarray*}
			\rho'_+(x, y) &=&\sup \{f(y): f \in J(x)\}= \sup \{ f(y): f \in Ext (J(x))\},\\
			\rho'_-(x, y) &=&\inf \{f(y): f \in J(x)\} = \inf  \{ f(y): f \in Ext( J(x))\}.
		\end{eqnarray*}
	\end{lemma}
	
	\begin{lemma}\cite{Amir}\label{B-J}
		Let $\mathbb{X}$ be a normed linear space. Then $x \perp_{\rho} y$ if and only if $\rho'_{-}(x, y) \leq 0 \leq \rho'_{+}(x, y).$ 
	\end{lemma}
For more information on $\rho$-orthogonality and related results, interested readers may see \cite{AGT, AST, CW, CW2, S}.
	It is a well known 	that $\perp_{\rho} \subset \perp_B.$
  For the converse inclusion we note the following result. 
	
	\begin{theorem}\label{smooth}\cite[Prop. 2.2.2]{AST}
		Let $\mathbb{X}$ be a normed linear space. Then for any $x, y \in \mathbb{X}$, $x \perp_B y$ implies $x \perp_\rho y$  if and only if   $\mathbb{X}$ is smooth. 
	\end{theorem}

	\begin{definition}
		Let $x \in \mathbb{X}.$ Then 
		\begin{itemize}
			 \item[(i)] $x$ is said to be  $(\rho_{+})$-left symmetric  if $x \perp_{\rho_{+}} y$ implies
			 that $y \perp_{\rho_{+}} x$  $~ \forall y \in \mathbb{X}$ and $x$ is said to be $(\rho_{+})$-right symmetric  if
			 $y \perp_{\rho_{+}} x$ implies that $x \perp_{\rho_{+}} y$ $ ~ \forall y \in  \mathbb{X}.$
			 
			 \item[(ii)]  $x $ is said to be $(\rho_{-})$-left symmetric  if $x \perp_{\rho_{-}} y$ implies
			 that $y \perp_{\rho_{-}} x$  $ ~ \forall y \in \mathbb{X}$ and  $ x \in \mathbb{X}$ is said to be  $(\rho_{-})$-right symmetric  if
			 $y \perp_{\rho_{-}} x$ implies that $x \perp_{\rho_{-}} y$  $ ~ \forall y \in  \mathbb{X}.$
			 
			 \item[(iii)]  $x $ is said to be  $\rho$-left symmetric if $x \perp_{\rho} y$ implies
			 that $y \perp_{\rho} x$  $~ \forall y \in \mathbb{X}$ and  $ x \in \mathbb{X}$ is said to be a $\rho$-right symmetric  if
			 $y \perp_{\rho} x$ implies that $x \perp_{\rho} y$  $ ~\forall y \in  \mathbb{X}.$
		\end{itemize}	
	\end{definition}
		The study of symmetric points concerning $\rho$-orthogonality was initiated in \cite{GPS}.\\
	

	For  a non-empty set $K$ and a Banach space $\mathbb{X},$ the Banach space of all bounded functions defined from $K$ to $\mathbb{X},$  endowed with the supremum norm, is denoted by $\ell_{\infty}(K, \mathbb{X}).$ Given a compact Hausdorff topological space $K$ and a Banach space $\mathbb{X}$, we write $C(K, \mathbb{X} )$ to denote the Banach space of all  continuous functions from $K$ to $\mathbb{X}$, endowed with the supremum norm. Clearly, $C(K,\mathbb{X})$ is  embedded into  $\ell_{\infty}(K, \mathbb{X}).$ Whenever $K$ is finite, it follows trivially that $\ell_{\infty}(K, \mathbb{X}) = C(K, \mathbb{X}).$ For the sake of simplicity, if $K$ is a finite set and $|K|=n,$ then $\ell_{\infty}(K, \mathbb{X})$ is denoted as $	\ell_{\infty}^n(\mathbb{X}).$
	For a function $f \in C(K, \mathbb{X}),$ the norm attainment set of $f,$ denoted by $M_f,$ is defined as $	M_f= \{k\in K: \|f(k)\|= \|f\|\}.$ Whenever $\mathbb{X}= \mathbb{R}$ or $ \mathbb{C},$ we use the standard    notations  $C(K)$   in place of $C(K, \mathbb{X})$.

	\section{Main Results.}
	
	\section*{\textbf{I. Symmetricity of BJ-orthogonality}}

In this section we study the symmetric elements with respect to Birkhoff-James orthogonality in the space of all bounded linear operators.

	\begin{theorem}\label{left:necessary}
	Let $\mathbb{X}$ be a reflexive Banach space such that $B_{\mathbb{X}}= co(Exp(B_{\mathbb{X}}))$ and $\mathbb{Y}$ be any Banach space. Suppose $T \in \mathbb{K}(\mathbb{X}, \mathbb{Y})$ is left symmetric. Then $T(x)$ is left symmetric, for any $x \in M_T \cap Exp(B_{\mathbb{X}}).$
\end{theorem}

\begin{proof}
	As $T$ is compact and $\mathbb{X}$ is reflexive, $M_T \neq \emptyset.$ Since $B_{\mathbb{X}}= co(Exp(B_{\mathbb{X}})),$ it is easy to observe that $M_T \cap Exp(B_{\mathbb{X}}) \neq \emptyset.$ Suppose on the contrary that $Tx$ is not left symmetric for some   $x \in M_T \cap Exp(B_{\mathbb{X}}).$ Then there exists $y_0 \in S_{\mathbb{Y}}$ such that $Tx \perp_B y_0$ but $y_0 \not\perp_B Tx.$ As $x$ is an exposed point of $B_{\mathbb{X}},$  there exists $f \in S_{\mathbb{X}^*}$ such that $f(x)=1$ and $f(u ) < 1,$ for any $u \in \mathbb{X} \setminus \{x\}.$ Define $S: \mathbb{X} \to \mathbb{Y}$ such that $Sz= f(z)y_0,$ for any $z \in \mathbb{X}.$ Clearly, $\|S\|=1$ and $M_S=\{ \pm x\}.$ Since $Tx \perp_B Sx,$ we observe that $\|T+ \lambda S\| \geq \|Tx+ \lambda Sx\|\geq \|Tx\|=\|T\|,$ for any scalar $\lambda.$ So, $T \perp_B S$. Since $M_S= \{\pm x\}$ and $Sx \not\perp_B Tx,$ it follows from \cite[Th. 2.1]{SP} that $S \not \perp_B T.$ This contradicts that $T$ is  left symmetric. This completes the proof of the theorem.
\end{proof}

As an immediate corollary of the above theorem we obtain the following result.

\begin{cor}
	Let $\mathbb{X}$ be a reflexive Banach space such that $B_{\mathbb{X}}= co(Exp(B_{\mathbb{X}})).$ Let $\mathbb{Y}$ be a Banach space such that there are no non-zero left symmetric points in $\mathbb{Y}.$ Then there exists no non-zero left symmetric operators in $\mathbb{K}(\mathbb{X}, \mathbb{Y})$.
\end{cor}

Applying Theorem \ref{left:necessary}, we characterize the left symmetric operators of the space $\mathbb{L}(\mathbb{X}, \ell_1^n).$ The proof is immediate as from \cite[Th. 2.8]{CSS}, there is no non-zero left symmetric points of $\ell_1^n,$ whenever $n \geq 3.$

\begin{cor}
	Let $\mathbb{X}$ be a reflexive Banach space such that $B_{\mathbb{X}}= co(Exp(B_{\mathbb{X}})).$ Then there exists no non-zero left symmetric operators in $\mathbb{L}(\mathbb{X}, \ell_1^n),$ where $n \geq 3.$.
\end{cor}

Note that whenever $\mathbb{X}$ is strictly convex Banach space, $Exp(B_{\mathbb{X}})= S_{\mathbb{X}}.$ Therefore the following result is immediate  as a corollary of Theorem \ref{left:necessary}.

\begin{cor}
		Let $\mathbb{X}$ be a reflexive strictly convex Banach space and $\mathbb{Y}$ be any Banach space. Suppose $T \in \mathbb{K}(\mathbb{X}, \mathbb{Y})$ is left symmetric. Then $Tx$ is left symmetric, for any $x \in M_T.$
\end{cor}

\begin{theorem}
	Let $\mathbb{X}, \mathbb{Y}$ be two reflexive Banach spaces.  Let $B_{\mathbb{Y}^*}= co(Exp(B_{\mathbb{Y}^*}))$.  Suppose $T \in \mathbb{K}(\mathbb{X}, \mathbb{Y})$ is left symmetric. Then $T^*(y^*)$ is left symmetric in $\mathbb{X}^*$, for any $y^* \in M_{T^*} \cap Exp(B_{\mathbb{Y}^*}).$
\end{theorem}

\begin{proof}
	As $T$ is compact, $T^*$ is also compact. Since $\mathbb{X}, \mathbb{Y}$ are reflexive and $T$ is left symmetric,  it is easy to show that $T^*$ is also left symmetric.  Following Theorem \ref{left:necessary}, $T^*(y^*)$ is left symmetric, for any $y^* \in M_{T^*} \cap Exp(B_{\mathbb{Y}^*}).$
\end{proof}

\begin{cor}
	Let $\mathbb{X}, \mathbb{Y}$ be two reflexive Banach spaces.  Let $B_{\mathbb{Y}^*}= co(Exp(B_{\mathbb{Y}^*}))$ and there exists no non-zero left symmetric points in $\mathbb{X}^*.$
	Then there exists no non-zero left symmetric operators in $\mathbb{K}(\mathbb{X}, \mathbb{Y})$.
\end{cor}

As $(\ell_{\infty}^n)^{*} = \ell_{1}^n$ and there exists no non-zero left symmetric points in $\ell_{1}^n,$ we next observe that there exists no   non-zero left symmetric operators in $\mathbb{L}(\ell_\infty^n, \mathbb{Y})$, where $n \geq 3$.

\begin{cor}
	Let $\mathbb{Y}$ be a reflexive Banach space with  $B_{\mathbb{Y}^*}= co(Exp(B_{\mathbb{Y}^*}))$. 
	Then there exists no non-zero left symmetric operators in $\mathbb{L }(\ell_\infty^n, \mathbb{Y})$, where $n \geq 3$.
\end{cor}

\begin{cor}
	Let $\mathbb{X}, \mathbb{Y}$ be  reflexive  Banach spaces such that $\mathbb{Y}$ is smooth.  Suppose $T \in \mathbb{K}(\mathbb{X}, \mathbb{Y})$ is left symmetric. Then $T^*(y^*)$ is left symmetric  in $\mathbb{X}^*$, for any $y^* \in M_{T^*}.$
\end{cor}

		\section*  {\textbf{II. $\rho$-orthogonality in function spaces and space of operators}}

		We begin this section with the following proposition.
		
		\begin{prop}\label{prop:l_infty}
			Let $K$ be a non-empty set and let   $f, g \in \ell_{\infty}(K, \mathbb{X} ).$ Consider  
		\[
		\Omega= \bigg\{ \lim y_n^*(g(k_n)): k_n \in K, y_n^* \in Ext (B_{\mathbb{X}^*}), \forall n \in \mathbb{N}, \lim y_n^*(f(k_n))= \|f\|\bigg\}.
		\]
			Then
		\begin{itemize}
			\item [(i)] $ \rho'_+(f,g)= \sup \Omega$
		\item [(ii)] $\rho'_-(f,g)= \inf \Omega$
		\item [(iii)] $\rho' (f,g)= \frac{1}{2}( \sup \Omega + \inf \Omega).$
		\end{itemize}  
		\end{prop}
		
		\begin{proof}
		(i) From Lemma \ref{functional},  $ \rho'_+(f,g)=\sup  \{ \psi(g): \psi \in S_{\mathbb{\ell_{\infty}(K, \mathbb{X})}^*}, \psi(f)=1\}.$ 
		Clearly,  $$ \rho'_+(f,g)=\sup  co( \{ \psi(g): \psi \in S_{\mathbb{\ell_{\infty}(K, \mathbb{X})}^*}, \psi(f)=1\})= \sup V(\ell_{\infty}(K, \mathbb{X}), f, g).$$
		  Consider the set 
		\[
		\mathcal{C}= \{ y^* \otimes \delta_k : k \in K, y^* \in Ext(B_{\mathbb{X}^*})\} \subset S_{\ell_{\infty}(K, \mathbb{X})^*},
		\]
		where $[y^* \otimes \delta_k ]( f ) := y^*(f(k)),$ for every $f \in \ell_{\infty}(K, \mathbb{X}).$ Observing that $B_{\ell_{\infty}(K, \mathbb{X})^*}= \overline{co(\mathcal{C} )}^{w*}$ (see \cite{MMQRS}) and using \cite[Cor. 2.5]{MMQRS}, we get
		\[
		V(\ell_{\infty}(K, \mathbb{X}), f, g)= co \bigg( \bigg\{ \lim y_n^*(g(k_n)): k_n \in K, y_n^* \in Ext(B_{\mathbb{X}^*}), \forall n \in \mathbb{N}, \lim y_n^*(f(k_n))= \|f\| \bigg\} \bigg).
		\]
		So, 
			\[
\rho'_+(f,g)	=\sup  \bigg\{ \lim y_n^*(g(k_n)): k_n \in K, y_n^* \in  Ext(B_{\mathbb{X}^*}), \forall n \in \mathbb{N}, \lim y_n^*(f(k_n))= \|f\| \bigg\}.
		\]

		(ii) follows similarly as (i) and (iii) follows as a consequence of (i) and (ii).
		 
		\end{proof}

		\begin{prop}\label{prop:continuous}
		 	Let $K$ be a compact Hausdorff space and let  $f, g \in C(K, \mathbb{X} ).$ Consider
		 \[
		 \Omega= \bigg\{  y^*(g(k)): k \in K, y^* \in  Ext(B_{\mathbb{X}^*}),   y^*(f(k))= \|f\|\bigg\}.
		 \]
		 Then
		 \begin{itemize}
		 	\item [(i)] $ \rho'_+(f,g)= \sup \Omega$
		 \item [(ii)] $\rho'_-(f,g)= \inf \Omega$
		 \item [(iii)] $\rho' (f,g)= \frac{1}{2}( \sup \Omega + \inf \Omega).$
		 \end{itemize}  
		\end{prop}

		\begin{proof}
	We only prove (i), as (ii) and (iii) follows in similar fashion.  We recall that (see \cite{R2}), $Ext (B_{C(K, \mathbb{X})^*})= \{  y^* \otimes \delta_k: k \in K, y^* \in Ext(B_{\mathbb{X}^*})\},$ 	where $[y^* \otimes \delta_k ]( f ) := y^*(f(k)),$ for  $f \in \ell_{\infty}(K, \mathbb{X}).$
	 Following Lemma \ref{functional}, 
	 \begin{eqnarray*}
	 	\rho_{+}'(f,g) &=& \sup \{ \psi(g): \psi \in Ext(J(f))\}\\
	 	&=& \sup \{ [ y^* \otimes \delta_k] (g): [ y^* \otimes \delta_k](f) = \|f\|, k \in K, y^* \in Ext(B_{\mathbb{X}^*})\}\\
	 	&=&  \sup \{ y^*(g(k)): k \in K, y^* \in Ext(B_{\mathbb{X}^*}) , y^*(f(k))= \|f\|\}\\& =&  \sup \Omega.
	 \end{eqnarray*}
		\end{proof}

		Following Proposition \ref{prop:l_infty}  and  \ref{prop:continuous} and the definition of $\rho$-orthogonality, ($\rho_{+}$)-orthogonality and ($\rho_{-}$)-orthogonality  the following results are immediate.
		
		\begin{theorem}\label{bounded:characterization}
		Let $K$ be a non-empty set and let  $f, g \in \ell_{\infty}(K, \mathbb{X} ).$ Consider 
		\[
		\Omega= \bigg\{ \lim y_n^*(g(k_n)): k_n \in K, y_n^* \in Ext(B_{\mathbb{X}^*}), \forall n \in \mathbb{N}, \lim y_n^*(f(k_n))= \|f\|\bigg\}.
		\]
		Then
		\begin{itemize}
			\item [(i)] $ f \perp_{\rho_+}  g$ if and only if $\sup \Omega=0.$
			\item [(ii)] $ f \perp_{\rho_-}  g$ if and only if $\inf \Omega=0.$
			\item [(iii)] $ f \perp_{\rho}  g$ if and only if $\sup \Omega +\inf \Omega=0.$
		\end{itemize}  
	
	\end{theorem}

		\begin{theorem}\label{continuous:characterization}
		Let $K$ be a  compact Hausdorff space and let  $f, g \in C(K, \mathbb{X} ).$ Consider   
		\[
		\Omega= \bigg\{  y^*(g(k)): k \in K, y^* \in Ext(B_{\mathbb{X}^*}),   y^*(f(k))= \|f\|\bigg\}.
		\]
		Then
		\begin{itemize}
			\item [(i)] $ f \perp_{\rho_+}  g$ if and only if $\sup \Omega=0.$
			\item [(ii)] $ f \perp_{\rho_-}  g$ if and only if $\inf \Omega=0.$
			\item [(iii)] $ f \perp_{\rho}  g$ if and only if $\sup \Omega + \inf \Omega=0.$
		\end{itemize}  
			\end{theorem}

As a consequence of Theorem  \ref{continuous:characterization}, we  characterize the $\rho_{+}$-orthogonality, $\rho_{-}$-orthogonality and $\rho$-orthogonality in the space $C(K).$
			
	\begin{cor}\label{C(K)}
			Let $K$ be a compact Hausdorff space. 	Let  $f, g \in C(K )$ and let 
			\[
			\Omega= \bigg\{  sgn(f(k)) g(k): k \in K,   |f(k)|= \|f\|\bigg\}.
			\]
			Then
			\begin{itemize}
				\item [(i)] $ f \perp_{\rho_+}  g$ if and only if $\sup \Omega=0.$
				\item [(ii)] $ f \perp_{\rho_-}  g$ if and only if $\inf \Omega=0.$
				\item [(iii)] $ f \perp_{\rho}  g$ if and only if $\sup \Omega + \inf \Omega=0.$
			\end{itemize}  
	\end{cor}		

The following two lemmas will be useful throughout this article.

	\begin{lemma}\label{lemma:1}
	Let $K$ be a  compact Hausdorff space and let $f \in S_{C(K, \mathbb{X})}.$ For any $g \in C(K, \mathbb{X})$,
	\begin{itemize}
		\item[(i)] if $f \perp_{\rho_+} g $ then there exists  $k_0 \in M_f$ such that $ f(k_0) \perp_{\rho_+} g(k_0)$
		\item[(ii)]  if $f \perp_{\rho_-} g $ then there exists  $k_0 \in M_f$ such that $ f(k_0) \perp_{\rho_-} g(k_0).$
	\end{itemize}
\end{lemma}

\begin{proof}
We only prove (i) as (ii) follows in similar fashion.  Let $f \perp_{\rho_+} g$. Following Theorem \ref{continuous:characterization}, we get $ \sup \Omega =0,$ where  
  $$\Omega = \{ y^*(g(k)): k \in K, y^* \in Ext(B_{\mathbb{X}^*}), y^*(f(k))= 1\}.$$  This implies  $y_n^*(g(k_n)) \to 0, $ where $k_n \in K, y_n^* \in Ext(B_{\mathbb{X}^*}) $ such that $y_n^*(f(k_n)) \to 1.$ As $K$ is compact, the sequence $\{k_n\}$ has a subnet $\{k_\lambda\}_{\lambda \in \Lambda} \subset \{k_n\}$ such that $k_\lambda \to k_0.$  
   Consider the net $\{y_\lambda^* \}_{\lambda \in \Lambda} \subset y_n^*$. As $B_{\mathbb{X}^*} $ is weak*-compact, the net  $\{y_\lambda^* \}_{\lambda \in \Lambda}$ has a subnet  $\{y_{\lambda_\alpha}^* \}$ such that $y_{\lambda_\alpha}^* \overset{w^*}{\longrightarrow} y_0^* \in B_{\mathbb{X}^*}.$ As $g$ is continuous, $g(k_{\lambda_\alpha}) \to g(k_0).$ Observe that 
   \begin{eqnarray*}
   	|y_{\lambda_\alpha}^*(g(k_{\lambda_\alpha})) -y_0^*(g(k_0))| &\leq& |y_{\lambda_\alpha}^*(g(k_{\lambda_\alpha})) - y_{\lambda_\alpha}^*(g(k_0)) + y_{\lambda_\alpha}^*(g(k_0))- y_0^*(g(k_0))| \\
   	& \leq & \|y_{\lambda_\alpha}^* \| \|(g(k_{\lambda_\alpha})) - g(k_0)\| + |y_{\lambda_\alpha}^*(g(k_0))- y_0^*(g(k_0))|
   \end{eqnarray*}
  Since  $y_{\lambda_\alpha}^* \overset{w^*}{\longrightarrow} y_0^*$,  we get $y_{\lambda_\alpha}^*(g(k_0)) \to y_0^*(g(k_0)).$ Therefore, $y_{\lambda_\alpha}^*(g(k_{\lambda_\alpha})) \to y_0^*(g(k_0))$ and $y_0^*(g(k_0)) =0.$  Since $f$ is continuous and $y_n^*(f(k_n)) \to 1$, proceeding similarly we get $y_0^*(f(k_0)) =1.$ This implies $k_0 \in M_f$ and $y_0^* \in J(f(k_0)).$ 
  So, $$ 0 \in \{y^*(g(k_0)): y^* \in J(f(k_0))\}.$$ Since $  \{y^*(g(k_0)): y^* \in J(f(k_0))\} \subset \Omega$ and $\sup \Omega=0,$ then $$\sup \{y^*(g(k_0)): y^* \in J(f(k_0))\}=0.$$ 
  This implies $f(k_0) \perp_{\rho_+} g(k_0).$
\end{proof}
			
			\begin{lemma}\label{lemma:singleton}
		Let $K$ be a  compact Hausdorff space and let $\mathbb{X}$ be a Banach space. Suppose that $f \in C(K, \mathbb{X})$ be such that $M_f=\{k_0\},$ for some $k_0 \in K.$ Then for any $g \in C(K, \mathbb{X})$,
		\begin{itemize}
			\item[(i)]  $f \perp_{\rho_+} g$ if and only if $f(k_0) \perp_{\rho_+} g(k_0)$
				\item[(ii)]  $f \perp_{\rho_-} g$ if and only if $f(k_0) \perp_{\rho_-} g(k_0)$
					\item[(iii)]  $f \perp_{\rho} g$ if and only if $f(k_0) \perp_{\rho} g(k_0).$
		\end{itemize}
			\end{lemma}
			
			\begin{proof}
(i)				For any $g \in C(K, \mathbb{X}),$ following Proposition \ref{prop:continuous}, we get
					$$ \rho_{+}'(f,g)= \sup \{ y^*(g(k)): k \in K, y^* \in Ext(B_{\mathbb{X}^*}), y^*(f(k))=\|f\|\}.$$
					Since $M_f=\{k_0\}$, it is immediate that 
				$$\rho_{+}'(f,g)= \sup \{ y^*(g(k_0)):  y^* \in Ext(B_{\mathbb{X}^*}), y^*(f(k_0))=\|f\|\}=\rho_{+}'(f(k),g(k)). $$
				Similarly, $\rho_{-}'(f,g)=\rho_{-}'(f(k),g(k)).$        
			 So $f \perp_{\rho_+} g \iff  \rho_{+}'(f,g) = 0 \iff  \rho_{+}'(f(k),g(k))=0    \iff f(k) \perp_{\rho_+} g(k). $ 
			
			(ii) and (iii) follows similarly.\\
			\end{proof}
		
		\section*{\textbf{Symmetricity of $\rho$-orthogonality in function spaces} }
		
		We begin this section with a sufficient condition for $(\rho_{+})$-symmetric elements in $\ell_{\infty}(K, \mathbb{X}).$
		
		\begin{theorem}\label{left:bounded}
			Let $K$ be a non-empty set and let $\mathbb{X}$ be a Banach space. Suppose that $f \in S_{\ell_{\infty}(K, \mathbb{X})}$  satisfies the following: 
			\begin{itemize}
				\item[(i)] there exists     $k_0 \in K$ such that $f(k_0) \in S_{\mathbb{X}}$ and $f(k)=0,~ \forall k (\neq k_0) \in K .$
				
				\item[(ii)] $f(k_0)$ is $(\rho_+)$-left symmetric.
			\end{itemize}   
			Then $f$ is $(\rho_+)$-left symmetric.
	\end{theorem}

		\begin{proof}
		Let $f \perp_{\rho_+} g$ and $\|g\|=1.$  To prove the theorem we need to show $g \perp_{\rho_+} f.$ For this applying Theorem \ref{bounded:characterization} we need to show that $\sup \Omega=0,$ where 
			\[
			\Omega=  \bigg\{ \lim y_n^*(f(k_n)): k_n \in K, y_n^* \in Ext(B_{\mathbb{X}^*}), \forall n \in \mathbb{N}, \lim y_n^*(g(k_n))= 1\bigg\}.
			\]
			Let $\{k_n\} \subset K$ be a sequence such that $\lim y_n^*(g(k_n))= 1.$ Then there are two possibilities either  $k_n \neq k_0,$ for all but finitely many $n \in \mathbb{N}$ or 	$\{k_n\}$ has a subsequence $\{k_{n_r}\}$ such that $k_{n_r}= k_0,$ for every $r \in \mathbb{N}.$ 
			For the first possibility, it is clear that $\lim y_n^*(f(k_n))=0$ as $f(k)=0, \forall k \neq k_0.$ 
			
			Otherwise, $1= \lim y_{n_r}^*(g(k_{n_r}))= \lim y_{n_r}^* (g(k_0)).$ Without loss of generality we assume that $ \lim y_n^*(g(k_0))=1.$
				Since $B_{\mathbb{X}^*}$ is weak*-compact, it follows that there exists a weak*-cluster point $y^* \in B_{\mathbb{X}^*}$ of the sequence $\{y_{n}^*\}.$ This implies the sequence $\{y_{n}^*(g(k_0))\}$ has a subsequence $\{y_{n_t}^* (g(k_0))\}$ such that $y_{n_t}^*(g(k_0)) \longrightarrow y^*(g(k_0)).$  This concludes that $y^*(g(k_0))=1= \|g\|.$ Therefore, $\|g(k_0)\|=\|g\|$ and $y^* \in J(g(k_0)).$ Proceeding similarly $\lim y_n^*(f(k_n))= y^*(f(k_0)).$ Therefore
				\begin{eqnarray*}
						\sup \Omega &=& \max\bigg\{0, \sup\{  y^*(f(k_0)):  y^* \in Ext(B_{\mathbb{X}^*}),  y^*(g(k_0))= 1 \} \bigg\} \\
						 &=& \max \bigg\{0, \sup\{  y^*(f(k_0)):  y^* \in J(g(k_0)) \} \bigg\}.
				\end{eqnarray*}
				We next show that $\sup\{  y^*(f(k_0)):  y^* \in J(g(k_0)) \}=0.$
	As $f \perp_{\rho_+} g$ and $f(k)=0, \forall k \neq k_0,$ applying Theorem \ref{bounded:characterization} it is easy to observe that $\sup \Omega_1 =0.$
				\[
			\Omega_1=\bigg\{ \lim z_n^*(g(k_0)): z_n^* \in Ext(B_{\mathbb{X}^*}), \forall n \in \mathbb{N}, \lim z_n^*(f(k_0))=1\bigg\}.
				\]
		Again using weak*-compactness of $B_{\mathbb{X}^*}$ and proceeding similarly, it is easy to observe that 
		\[
	\sup \Omega_1=	\sup \{ z^*(g(k_0)): z^* \in Ext(B_{\mathbb{X}^*}),  z^*(f(k_0))=1 \}=0.
		\]	
		This implies that $
			\rho_{+}'(f(k_0), g(k_0))=	0.$
			Following Lemma \ref{functional}, we get $f(k_0) \perp_{\rho_+} g(k_0).$ Since $f(k_0)$ is $(\rho_+)$-left symmetric, we get $g(k_0) \perp_{\rho_+} f(k_0)$. Again following Lemma \ref{functional}, we get  $$  \sup\{  y^*(f(k_0)):  y^* \in J(g(k_0)) \}= \rho_{+}'(g(k_0), f(k_0))=0.$$ Therefore we obtain   $\sup \Omega=0.$ This completes the theorem. \\
			
			
		\end{proof}

		As the space $C(K, \mathbb{X})$ is a subspace of $\ell_{\infty}(K, \mathbb{X}),$ so if an element $f \in C(K, \mathbb{X})$ is $(\rho_{+})$-left symmetric in $\ell_{\infty}(K, \mathbb{X}),$ then it will be naturally left symmetric in $C(K, \mathbb{X}).$
			Thus if a continuous function $f$ satisfies the sufficient condition of Theorem \ref{left:bounded}, then $f$ is also $(\rho_+)$-left symmetric  in $C(K, \mathbb{X}).$ In the next theorem, we show that the conditions are necessary for an element  $f \in C(K, \mathbb{X})$ to be $(\rho_{+})$-left symmetric under the additional condition that $K$ is perfectly normal. 
		 Let us recall that a topological space $K$ is said to be perfectly normal if  $K$ is normal and every closed set of $K$ is a $G_\delta$ set.
		
		\begin{theorem}\label{left:continuous}
			Let $K$ be a compact, perfectly normal space and let $\mathbb{X}$ be a Banach space. Then   $f \in S_{C(K, \mathbb{X})}$  is $(\rho_+)$-left symmetric if and only if $f$ satisfies the following : 
			\begin{itemize}
				\item[(i)] there exists  $k_0 \in K$ such that $f(k_0) \in S_{\mathbb{X}}$ and  $f(k)=0,$ $\forall k_0 \neq k \in K .$
				\item[(ii)] $f(k_0) $ is $(\rho_+)$-left symmetric.
			\end{itemize} 
			
		\end{theorem}

		\begin{proof}
			Since the sufficient part follows from Theorem \ref{left:bounded}, we prove only the necessary part. 
			
		As $f \in S_{C(K, \mathbb{X})},$ $M_f$ is non-empty. Let $k_0 \in M_f.$ We claim that $f(k)=0, ~\forall k \neq k_0.$ Suppose on the contrary there exists $k_1 \neq k_0 $ such that    $f(k_1) \neq 0.$
		Let $V= \{ x \in \mathbb{X}: \|f(k_1)-x\| < \frac{1}{2}\}.$ Clearly, $V$ is open. As $f$ is continuous, there exists an open set $U$ containing $k_1$ such that $f(U) \subset  V.$  
			 As $K$ is   Hausdorff,  we can assume $k_0 \notin U$.  Since $K$ is perfectly normal and $\{k_1\},   K \setminus U $ are two disjoint closed sets in $K$, then using Urysohn's lemma (see \cite{Munkres}), there exists a continuous function $\alpha: K \to [0,1]$ such that $\alpha^{-1}(\{1\})= \{k_1\}$ and $\alpha^{-1}(\{0\})= K \setminus U.$ Define $g: K \to \mathbb{X}$ as  $$g(k)= -  \alpha(k)\frac{f(k_1)}{ \|f(k_1) \|}  ,~ \forall k \in K.$$  Clearly, $g$ is continuous. 
		Consider 
		\[
		\Omega_1= \bigg\{  y^*(g(k)): k \in K, y^* \in Ext(B_{\mathbb{X}^*}),   y^*(f(k))= \|f\|\bigg\}.
		\]
		Observe that $g(k)=0, \forall k \in K \setminus U.$ So, 
		\[
		\sup \Omega_1= \max \bigg\{ 0, \sup\{ y^*(g(k)): k \in U, y^* \in Ext(B_{\mathbb{X}^*}),   y^*(f(k))= \|f\| \}\bigg\}.
		\]
	Observe that $y^*(g(k)) = -\alpha(k) \frac{1}{\|f(k_1)\|} y^*(f(k_1)). $ 	Let $k \in U$ and $y^* \in Ext(B_{\mathbb{X}^*})$ such that $  y^*(f(k))= \|f\| .$
	 Then
	 $$ |y^*(f(k_1))-1| =|y^*(f(k_1)) -y^*(f(k))| \leq \|f(k_1)- f(k)\| < \frac{1}{2}.  $$
		So, $y^*(f(k_1)) > \frac{1}{2}.$  Clearly, $y^*(g(k)) < 0,$ for any $ k \in U $ and for any  $y^* \in Ext(B_{\mathbb{X}^*})$ such that $  y^*(f(k))= \|f\| .$ Therefore, 
		$$\sup\{ y^*(g(k)): k \in U, y^* \in Ext(B_{\mathbb{X}^*}),   y^*(f(k))= \|f\| \} \leq 0. $$
		 Hence $\sup \Omega_1=0.$ Following Theorem \ref{continuous:characterization}, we get $f \perp_{\rho_+} g.$
		
		On the other hand, as $\alpha^{-1} (\{1\})=\{k_1\}, \alpha^{-1}(\{0\})= K \setminus U,$ so  we have $g(k_1)=- \frac{f(k_1)}{ \|f(k_1) \|} ,$   and  $g(k)=0,$ $\forall k \in K \setminus U.$ Moreover, $\|g\|=1$ and $M_g=\{k_1\}.$ Observe that whenever $y^* \in J(g(k_1)),$ we get $y^*(f(k_1))= - \|f(k_1)\|.$ So, 
		\[
		\bigg\{  y^*(f(k)): k \in K, y^* \in Ext(B_{\mathbb{X}^*}),   y^*(g(k))= \|g\|\bigg\} = \{- \|f(k_1)\|\}.
		\] 
		As $\|f(k_1)\| \neq 0,$ again following Theorem \ref{continuous:characterization}, $g \not \perp_{\rho_+} f.$ This contradicts the fact that $f$ is $(\rho_{+})$-left symmetric. Thus  $f$ satisfies the condition (i).\\
		
	Next we show that $f$ satisfies the condition (ii). Suppose on the contrary that $f(k_0)$ is not $(\rho_+)$-left symmetric. This implies that there exists $w \in S_\mathbb{X}$ such that $ f(k_0) \perp_{\rho_+} w$ and $w \not \perp_{\rho_+} f(k_0)$.  Let $U \subset K$ be an open set containing $k_0.$ As $K$ is perfectly normal and $\{k_0\},   K \setminus U $ are two disjoint closed sets in $K$, so by the Urysohn's lemma (see \cite{Munkres}), there exists a continuous function $\alpha: K \to [0,1]$ such that $\alpha^{-1}(\{1\})= \{k_0\}$ and $\alpha^{-1}(\{0\})= K \setminus U.$ Define $g: K \to \mathbb{X}$ as  $$g(k)= -  \alpha(k)w  ,~ \forall k \in K.$$  Clearly, $g$ is continuous.  As $f$ satisfies the condition (i), so $M_f=\{k_0\}$ and  since $f(k_0) \perp_{\rho_+} g(k_0),$  
		following Lemma \ref{lemma:singleton}, we get $f \perp_{\rho_+} g.$   Moreover,    $M_{g}= \{k_0\}$ and $g(k_0) \not \perp_{\rho_+} f(k_0)$. Again 
	using Lemma \ref{lemma:singleton}, we get $g \not \perp_{\rho_+} f.$ This contradicts the fact that $f$ is $(\rho_+)$-left symmetric.
		\end{proof}

			As the space $C(K, \mathbb{X})$ is a subspace of $\ell_{\infty}(K, \mathbb{X}),$ therefore to be a $(\rho_+)$-left symmetric, a bounded function in $\ell_{\infty}(K, \mathbb{X})$ has to satisfy the necessary conditions of Theorem \ref{left:continuous}. Therefore using Theorem \ref{left:bounded} and the necessary conditions of Thoerem \ref{left:continuous} we obtain the complete characterization of $(\rho_+)$-left symmetric functions in $\ell_{\infty}(K, \mathbb{X}).$

		\begin{theorem}\label{left:infty}
			Let $K$ be a non-empty set and let $\mathbb{X}$ be a Banach space. Then $f \in S_{\ell_{\infty}(K, \mathbb{X})}$ is $(\rho_+)$-left symmetric if and only if $f$ satisfies the following: 
		\begin{itemize}
			\item[(i)] there exists a  $k_0 \in K$ such that $f(k_0) \in S_{\mathbb{X}}$ and $f(k)=0, \forall k_0 \neq k \in K .$
			
			\item[(ii)] $f(k_0)$ is $(\rho_+)$-left symmetric.
		\end{itemize}   
		\end{theorem}
	
	As an immediate corollary of Theorem \ref{left:continuous} and \ref{left:infty}, we obtain the complete characterization of $(\rho_{+})$-left symmetric points of $C_0(K)$ and $\ell_{\infty}.$

		\begin{cor}
				Let $K$ be a  compact, perfectly normal space. Then $f \in S_{C(K)}$ is $(\rho_+)$-left symmetric if and only if $M_f= \{k_0\}$,
 for some $k_0 \in K$ and $f(k)=0, \forall k \neq k_0.$	
 	\end{cor}
 	
 	\begin{cor}
 		Let $\widetilde{x}=(x_1, x_2, \ldots) \in S_{\ell_{\infty}}$. Then $\widetilde{x}$ is $(\rho_+)$-left symmetric if and only if $|x_k|=1,$ for some $k \in \mathbb{N}$ and $x_i=0, \forall i \neq k.$
 	\end{cor}
		
	\begin{theorem}\label{right:continuous}
			Let $K$ be a  compact, perfectly normal space. Then   $f \in S_{C(K, \mathbb{X})}$  is $(\rho_+)$-right symmetric if and only if		$f$ satisfies exactly one of the following: 
		\begin{itemize}
			\item[(i)]
	If	 $M_f= \{k_0\}$ then  $f(k_0)$ is $(\rho_{+})$-right symmetric  and  for any $k \neq k_0,$  $w \perp_{\rho_+} f(k) \implies w=0.$ 
		
		\item[(ii) ] Otherwise  for any $k \in K$,   $w \perp_{\rho_+} f(k) \implies w=0$.
		\end{itemize} 
	\end{theorem}	
		
\begin{proof}
	         Let us prove the necessary part first.  \\
	         Let $M_f = \{k_0\}.$ Suppose on the contrary that $f(k_0)$ is not $(\rho_{+})$-right symmetric. This implies there exists $w \in S_{\mathbb{X}}$ such that $w \perp_{\rho_+} f(k_0)$ and $f(k_0) \not \perp_{\rho_+} w.$ Suppose that $U$ is an open set containing $k_0$.    As $K$ is perfectly normal and $\{k_0\},   K \setminus U $ are two disjoint closed sets in $K$, so by the Urysohn's lemma (see \cite{Munkres}), there exists a continuous function $\alpha: K \to [0,1]$ such that $\alpha^{-1}(\{1\})= \{k_0\}$ and $\alpha^{-1}(\{0\})= K \setminus U.$ Define $g: K \to \mathbb{X}$ such that 
	         \[
	         g(k)= \alpha(k) w,  \quad \forall k \in K.
	         \] 
	         Clearly, $g \in C(K, \mathbb{X}).$ Moreover, $M_g= \{k_0\}$ and $g(k_0)=w.$ As $ g(k_0) \perp_{\rho_+} f(k_0)$, applying Lemma \ref{lemma:singleton} we get $g \perp_{\rho_+} f.$ Again $M_f= \{k_0\}$ and $f(k_0) \not \perp_{\rho_+} g(k_0)$, so following Lemma \ref{lemma:singleton}, we get $f \not \perp_{\rho_+} g.$ This contradicts that $f$ is $(\rho_{+})$-right symmetric.

	       Again   suppose on the contrary there exist $k' \in K$ and $w \in S_{\mathbb{X}}$ such that $\|f(k')\| < 1$ and  $w \perp_{\rho_+} f(k').$ Let $U$ and $V$ two disjoint open sets such that $k' \in U$ and $ M_f \subset V.$ 
	        	  As $K$ is perfectly normal and $\{k'\},   K \setminus U $ are two disjoint closed sets in $K$, so by the Urysohn's lemma (see \cite{Munkres}), there exists a continuous function $\alpha: K \to [0,1]$ such that $\alpha^{-1}(\{1\})= \{k'\}$ and $\alpha^{-1}(\{0\})= K \setminus U.$ Again considering the two disjoint closed sets $M_f$ and $k \setminus V,$ we can find 
	        	  a continuous function $\beta: K \to [0,1]$ such that $\beta^{-1}(\{1\})= M_f$ and $\beta^{-1}(\{0\})= K \setminus V.$ 
	        Define $g: K \to \mathbb{X}$ such that $$g(k)= \alpha(k) w+ \frac{1}{2} \beta(k) f(k), \quad \forall k \in K.$$ 
	         Clearly, $g \in C(K, \mathbb{X}).$ 
	         It is easy to see that $M_g = \{ k'\}$ and $g(k') =w.$ As $w \perp_{\rho_+} f(k'),$ from Lemma \ref{lemma:singleton}, $g \perp_{\rho_+} f.$
	         Observe that for $k \in M_f,$ $g(k)= \frac{1}{2} f(k).$ So, 
	         \[
	         \{y^*(g(k)): k \in K, y^* \in Ext(B_{\mathbb{X}^*}), y^*(f(k))= 1\} = \bigg\{ \frac{1}{2}\bigg\}.
	         \]
	    Following Theorem \ref{continuous:characterization}, we get $f \not \perp_{\rho_+} g.$ This contradicts the fact that $f$ is $(\rho_{+}       )$-right symmetric.\\ 

	    Let us now assume that $M_f$ is not singleton. Suppose on the contrary that there exists $k_1 \in K$ and $w \in S_{\mathbb{X}}$ such that $w \perp_{\rho_+} f(k_1)$. Suppose $ k_2 \in M_f  \setminus \{k_1\}$ and let $U, V$ two disjoint open sets containing $k_1, k_2,$ respectively.  As $K$ is perfectly normal and $\{k_1\},   K \setminus U $ are two disjoint closed sets in $K$, so by the Urysohn's lemma (see \cite{Munkres}), there exists a continuous function $\alpha: K \to [0,1]$ such that $\alpha^{-1}(\{1\})= \{k_1\}$ and $\alpha^{-1}(\{0\})= K \setminus U.$ Similarly, there exists a continuous function $\beta: K \to [0,1]$ such that $\beta^{-1}(\{1\})= \{k_2\}$ and $\beta^{-1}(\{0\})= K \setminus V.$  Define $g : K \to \mathbb{X}$ such that 
	    \[
	    g(k)= \alpha(k) f(k_1) + \frac{1}{2} \beta(k) f(k_2), \quad \forall k \in K.  
	    \]
	     Clearly, $g \in C(K, \mathbb{X}).$ Moreover, $M_g= \{k_1\}$ and $g(k_1)=w.$ As $w \perp_{\rho_+} f(k)$, applying Lemma \ref{lemma:singleton}, we get $g \perp_{\rho_+} f.$ As $k_2 \in M_f,$
	     \[
	     \frac{1}{2} \in \bigg\{ y^*(g(k)): k \in K, y^* \in Ext(B_{\mathbb{X}^*}), y^*(f(k)) = 1\bigg\}.
	     \]
	    Following Theorem \ref{continuous:characterization}, we get  $f \not \perp_{\rho_+} g.$ This contradicts the fact that $f$ is $(\rho_{+})$-right symmetric. This completes the necessary part.\\

	    To complete the theorem we need to prove the sufficient part. First consider that $f \in S_{C(K, \mathbb{X})}$ satisfies the condition (i). Let $M_f=\{k_0\}$, $f(k_0)$ is $(\rho_{+})$-right symmetric and for any $k \notin M_f$, $ w \perp_{\rho_+} f(k) \implies w=0.$    
	     Let $g \perp_{\rho_+} f.$ Following Lemma \ref{lemma:1}, there exists $k_1 \in M_g$ such that  $g(k_1) \perp_{\rho_+}  f(k_1).$
	        Following the hypothesis we have $k_1=k_0.$
	       As $M_f=\{k_0\}$ and   $f(k_0)$ is $(\rho_{+})$-right symmetric. So, we have $f(k_0) \perp_{\rho_+} g(k_0).$ Following Lemma \ref{lemma:singleton}, we get $f \perp_{\rho_+} g$ and hence $f$ is $(\rho_{+})$-right symmetric. 
	       
	       Next we consider that $f \in S_{C(K, \mathbb{X})}$ satisfies the condition (ii).  Let $g \perp_{\rho_+} f.$ Following Lemma \ref{lemma:1}, there exists $k_1 \in M_g$ such that  $g(k_1) \perp_{\rho_+}  f(k_1).$ From hypothesis it implies $g(k_1)=0.$ Since $k_1 \in M_g$ we get $g=0.$ Thus  $ g \perp_{\rho_+} f \implies f \perp_{\rho_+} g$ and so $f$ is $(\rho_{+})$-right symmetric.\\
	     \end{proof}

	\begin{cor}
	Let $K$ be a compact, perfectly normal space. Then $f \in S_{C(K)}$ is $(\rho_+)$-right symmetric if and only if $f(k) \neq 0,$ for any $k \in K.$
\end{cor}		
		
			\begin{cor}
			Let $K$ be a compact, perfectly normal space. Then there exists no non-zero $(\rho_{+})$-symmetric element in $C(K, \mathbb{X}).$
		\end{cor}

	Next we consider the $(\rho_-)$-left (right) symmetricity in  similar domains. The spirit of proofs are almost same as above and so,  we omit detailed proofs. However, it is important to highlight a key difference: for the ($\rho_{-}$)-left (right) symmetricity, the analysis involves considering the infimum of the relevant sets instead of the supremum. 
	We should also mention that to do the necessary parts of the theorem, we need to construct the function $g$ analogously, except with  a change of sign of the functions.

			\begin{theorem}\label{left2:continuous}
			Let $K$ be a  compact, perfectly normal space. Then   $f \in S_{C(K, \mathbb{X})}$  is $(\rho_-)$-left symmetric if and only if $f$ satisfies the following : 
			\begin{itemize}
				\item[(i)] there exists  $k_0 \in K$ such that $f(k_0) \in S_{\mathbb{X}}$ and  $f(k)=0,$ $\forall k_0 \neq k \in K .$
				\item[(ii)] $f(k_0) $ is $(\rho_-)$-left symmetric.
			\end{itemize} 
			
		\end{theorem}
		
			\begin{theorem}
			Let $K$ be a non-empty set. Then $f \in S_{\ell_{\infty}(K, \mathbb{X})}$ is $(\rho_-)$-left symmetric if and only if $f$ satisfies the following: 
			\begin{itemize}
				\item[(i)] there exists   $k_0 \in K$ such that $f(k_0) \in S_{\mathbb{X}}$ and $f(k)=0, \forall k_0 \neq k \in K .$
				
				\item[(ii)] $f(k_0)$ is $(\rho_-)$-left symmetric.
			\end{itemize}   
		\end{theorem}

			\begin{theorem}\label{right2:continuous}
			Let $K$ be a  compact, perfectly normal space.  Then   $f \in S_{C(K, \mathbb{X})}$  is $(\rho_-)$-right symmetric if and only if			$f$ satisfies exactly one of the following: 
		\begin{itemize}
			\item[(i)]
			If	 $M_f= \{k_0\}$ then  $f(k_0)$ is $(\rho_{-})$-right symmetric  and  for any $k \notin M_f,$   $w \perp_{\rho_-} f(k) \implies w =0.$ 
			
			\item[(ii)] Otherwise  for any $k \in K$,   $w \perp_{\rho_-} f(k) \implies w =0$.
		\end{itemize} 
		\end{theorem}


	Next we study the symmetricity of $\rho$-orthogonality. First we present a necessary condition for $\rho$-left symmetric points.	  
		
		\begin{theorem}\label{left:rho:continuous}
		Let $K$ be a  compact, perfectly normal space.     Suppose  $f \in S_{C(K, \mathbb{X})}$  is $\rho$-left symmetric. Then $f(k)=0, ~ \forall k \notin M_f.$
		\end{theorem}
		
		\begin{proof}
				Suppose on the contrary we assume that $f(k_1) \neq 0$ and $\|f(k_1)\| < 1,$ for some $k_1 \in K.$ As $f$ is continuous, there exists an open set $U$ of $K$ containing $k_1$ such that $\|f(k)\| < 1, ~\forall k \in U$.
		 As $K$ is perfectly normal and $\{k_1\},   K \setminus U $ are two disjoint closed sets in $K$, so by the Urysohn's lemma (see \cite{Munkres}), there exists a continuous function $\alpha: K \to [0,1]$ such that $\alpha^{-1}(\{1\})= \{k_1\}$ and $\alpha^{-1}(\{0\})= K \setminus U.$ Define $g: K \to \mathbb{X}$ as  $$g(k)=   \alpha(k)\frac{f(k_1)}{ \|f(k_1) \|}  ,~ \forall k \in K.$$  Clearly, $g$ is continuous.
		  Observe that for any $k \in M_f,$ $g(k)=0.$ So, 
		 \[
		 \Omega= \bigg\{  y^*(g(k)): k \in K, y^* \in Ext(B_{\mathbb{X}^*}),   y^*(f(k))= 1 \bigg\} = \{0\}.
		 \]
		 Therefore, $\sup \Omega+ \inf \Omega=0.$ Following Theorem \ref{continuous:characterization}, we get $f \perp_{\rho} g.$
		 
			On the other hand, as $\alpha^{-1} (\{1\})=\{k_1\}, \alpha^{-1}(\{0\})= K \setminus U,$ so  we have $g(k_1)=  \frac{f(k_1)}{ \|f(k_1) \|} ,$   and  $g(k)=0,$ $\forall k \in K \setminus U.$ Moreover, $\|g\|=1$ and $M_g=\{k_1\}.$ Observe that whenever $y^* \in J(g(k_1)),$ we get $y^*(f(k_1))=  \|f(k_1)\|.$ So, 
		\[
	\Omega_1=	\bigg\{  y^*(f(k)): k \in K, y^* \in Ext(B_{\mathbb{X}^*}),   y^*(g(k))= \|g\|\bigg\} = \{ \|f(k_1)\|\}.
		\] 
		So, $\sup \Omega_1+ \inf \Omega_1 \neq 0.$ Following Theorem \ref{continuous:characterization}, we get $g \not \perp_{\rho} f.$ This contradicts the fact that $f$ is $\rho$-left symmetric. This completes the proof. 
		
		 	\end{proof}
				 	
	\begin{cor}
		Let $K$ be a compact, connected, perfectly normal space. Suppose that $f \in C(K, \mathbb{X})$ is  a $\rho$-left symmetric function. Then either $f=0$ or $M_f= K.$
			\end{cor}

			\begin{theorem}\label{right:rho:continuous}
			Let $K$ be a  compact, perfectly normal space. Suppose that   $f \in S_{C(K, \mathbb{X})}$ is $\rho$-right symmetric. 	Then  $f$ satisfies the following: 
			\begin{itemize}
				\item[(i)] $f(k) \neq 0,$ for any $k \in K$
				\item[(ii)] $\|f(k_1)\| = \|f(k_2)\|$ implies that $\|f(k_1)\|=\|f(k_2)\|=1$
				\item[(iii)] $\|f(k)\|=1$ implies that $f(k)$ is  $\rho$-right symmetric.
			\end{itemize} 
		\end{theorem}

	\begin{proof}
		We first show that $f$ satisfies the condition (i).
			Suppose on the contrary we first assume that $f(k_0) =0,$ for some $k_0 \in K.$ Since $K$ is  compact and  Hausdorff, there exist two disjoint open sets $U, V$ of $K$ such that $k_0 \in U$, $M_f \subset V$. 
	  As $K$ is perfectly normal and $\{k_0\},   K \setminus U $ are two disjoint closed sets in $K$, so by the Urysohn's lemma (see \cite{Munkres}), there exists a continuous function $\alpha: K \to [0,1]$ such that $\alpha^{-1}(\{1\})= \{k_0\}$ and $\alpha^{-1}(\{0\})= K \setminus U.$ Since $M_f$ is also compact, similarly there exists a continuous function $\beta: K \to [0,1]$ such that $\beta^{-1}(\{1\})= M_f$ and $\beta^{-1}(\{0\})= K \setminus V.$ 
	 Let $x \in S_{\mathbb{X}}.$ Define $g: K \to \mathbb{X}$ such that $$g(k)= \alpha(k) x+ \frac{1}{2} \beta(k) f(k), \quad \forall k \in K.$$ 
	 Clearly,  $g \in C(K, \mathbb{X}).$ 
	 It is easy to see that $M_g = \{ k_0\}$ and $g(k_0) =x.$ As $x \perp_{\rho} f(k_0),$ from Lemma \ref{lemma:singleton}, we get $g \perp_{\rho} f.$
		 Observe that for any $k \in M_f,$ $g(k)= \frac{1}{2} f(k).$ So, 
	 \[
	 \{y^*(g(k)): k \in K, y^* \in Ext(B_{\mathbb{X}^*}), y^*(f(k))= 1\} = \bigg\{ \frac{1}{2}\bigg\}.
	 \]
	 Following Theorem \ref{continuous:characterization}, we get $f \not \perp_\rho g.$ This contradicts the fact that $f$ is $\rho$-right symmetric. Therefore $f$ satisfies the condition (i).\\
	 
	 We next prove that $f$ satisfies the condition (ii). Again on the contrary we assume that $\|f(k_1)\|= \|f(k_2)\|= r< 1,$ for some $k_1, k_2 \in K.$ Since $K$ is  compact and Hausdorff, there exist mutually disjoint open sets $U, V, W$ of $K$ such that $k_1 \in U, k_2 \in V, M_f \subset W$.  As $K$ is perfectly normal, using Uryshon's Lemma there exists $\alpha, \beta, \gamma: K \to [0,1]$ such that 
	 \begin{eqnarray*}
	 	\alpha^{-1} (\{1\})= \{ k_1\},\, \,  \alpha^{-1}(\{0\})= K \setminus U,\\
	 	\beta^{-1} (\{1\})= \{ k_2\}, \, \, \,  \beta^{-1}(\{0\})= K \setminus V,\\
	 	\gamma^{-1} (\{1\})= M_f, \quad  \gamma^{-1}(\{0\})= K \setminus W.
	 \end{eqnarray*}
	 Define $g : K \to \mathbb{X}$ such that 
	 \[
	 g(k)= \alpha(k) \frac{f(k_1)}{r} - \beta(k) \frac{f(k_2)}{r}+ \frac{1}{2} \gamma(k) f(k), ~~~ \forall k \in K.
	 \]
	  Clearly, $g \in C(K, \mathbb{X}).$ 
		 Let $$\Omega = \bigg\{ y^*(f(k)): k \in K, y^* \in Ext(B_{\mathbb{X}^*}), y^*(g(k))=1 \bigg\}.$$
	  It is easy to see that $M_g = \{ k_1, k_2\}.$  Therefore, 
	  \[
	  \Omega = \{ y^*(f(k_1)):  y^* \in Ext(B_{\mathbb{X}^*}), y^*(g(k_1))=1 \} \cup \{ y^*(f(k_2)):  y^* \in  Ext(B_{\mathbb{X}^*}), y^*(g(k_2))=1 \}.
	  \]
	  As  $g(k_1)= -\frac{f(k_1)}{r}, g(k_2)= \frac{f(k_2)}{r},$ it is easy to observe that $\Omega= \{-r, r\}.$ So, $\sup \Omega+ \inf \Omega =0.$ Following Theorem \ref{continuous:characterization}, we get $g \perp_{\rho} f$.     \\
	  Observe that for any $k \in M_f,$ $g(k)= \frac{1}{2} f(k).$ So, it is easy to observe that
	 \[
	 \{y^*(g(k)): k \in K, y^* \in Ext(B_{\mathbb{X}^*}), y^*(f(k))= 1\} = \bigg\{ \frac{1}{2}\bigg\}.
	 \]
	 Following Theorem \ref{continuous:characterization}, we get $f \not \perp_{\rho} g.$ This contradicts that $f$ is $\rho$-right symmetric. Therefore $f$ satisfies the condition (ii).\\
	 
	 To complete the theorem we now need to show that $f$ satisfies (iii). Suppose on the contrary that $\|f(k_0)\|=1$ and $f(k_0)$ is not $\rho$-right symmetric. Then there exists $w \in S_{\mathbb{X}}$ such that $w \perp_{\rho} f(k_0)$ and $f(k_0) \not  \perp_{\rho} w.$ Let $$\Gamma= \{ y^*(w): y^* \in Ext(J(f(k_0)))\}.$$
	 So, $ \sup \Gamma + \inf \Gamma \neq 0.$ Let us consider the following cases: 
	  
	  \textbf{Case~1:} 
	  Suppose that $\sup \Gamma + \inf \Gamma  > 0.$ Let $\sup \Gamma + \inf \Gamma = r > 0.$
	  Clearly. $\sup \Gamma \geq \frac{r}{2}.$
	For any $v^* \in Ext(J(f(k_0))) ,$   consider the set 
	   $$G_{v^*}= \{ x^* \in S_{\mathbb{X}^*}: |(x^* - v^*) w| < \frac{r}{2} \}.$$
	    Clearly, $G_{v^*}$ is an weak* open subset of $S_{\mathbb{X}^*}.$ Let \[\mathcal{G}= \underset{v^* \in Ext(J(f(k_0)))}{\bigcup} G_{v^*}. \]
	   Observe that  the set $V= \{ x \in \mathbb{X}: x \in S_{\mathbb{X}}, J(x) \subseteq \mathcal{G}\}$ is an open set of $S_{\mathbb{X}}.$ Moreover, $V \cap f(K)$ is non-empty and $V \cap f(K)$ is   an open subset of $f(K) \cap S_{\mathbb{X}}.$ Consider $U = f^{-1} (V \cap f(K)).$ As $f$ is continuous, $U$ is open in $K.$ Clearly, $k_0 \in U.$
	  As $K$ is perfectly normal and $\{k_0\},   K \setminus U $ are two disjoint closed sets in $K$, so by the Urysohn's lemma (see \cite{Munkres}), there exists a continuous function $\alpha: K \to [0,1]$ such that $\alpha^{-1}(\{1\})= \{k_0\}$ and $\alpha^{-1}(\{0\})= K \setminus U.$ Define $g : K \to \mathbb{X}$ such that $g(k)= \alpha(k)w, \forall k \in K.$   Clearly, $g \in C(K, \mathbb{X}).$ 
	  It is easy to see that $M_g = \{ k_0\}$ and $g(k_0) =w.$ As $w \perp_{\rho} f(k_0),$ from Lemma \ref{lemma:singleton}, we get $g \perp_{\rho} f.$
	  Let $$\Omega = \bigg\{ y^*(g(k)): k \in K, y^* \in Ext(B_{\mathbb{X}^*}), y^*(f(k))=1 \bigg\}.$$
	  As $g(k)=0, ~\forall k \notin U,$ so
	  $$\Omega \subseteq \bigg\{ \alpha(k) y^*(w): k \in U, y^* \in Ext(B_{\mathbb{X}^*}), y^*(f(k))=1 \bigg\} \cup \{0\}.$$
	   As $k_0 \in U$ and $\alpha(k_0)=1,$ therefore for any $y^* \in Ext( J(f(k_0))),$ we have $y^*(w) \in \Omega.$ Thus $\Gamma \subset \Omega$ and so, $\sup \Omega \geq  \sup \Gamma \geq \frac{r}{2} > 0.$\\
	Suppose $\inf \Gamma = u^*(w),$ for some $u^* \in J(f(k_0)).$  For any $k \in U ,$ $J(f(k) \subset \mathcal{G}.$ Let $k \in U$ and $y^* \in J(f(k))$.
 Suppose that  $y^* \in G_{v^*},$ for some $v^* \in Ext(J(f(k_0))).$ This implies $|y^*(w)- v^*(w)| < \frac{r}{2}.$
	So, we have 
	\[
	u^* (w) - \frac{r}{2} \leq v^*(w) - \frac{r}{2} < y^*(w)< v^*(w) + \frac{r}{2}.
	\]
	When $y^*(w) \geq 0$  we have $\alpha(k)y^*(w) \geq 0$ and  
whenever $y^*(w)< 0$ we have $\alpha(k)y^*(w) \geq y^*(w).$ 
Therefore, 
\[
\alpha(k)y^*(w) \geq y^*(w) > 	u^* (w) - \frac{r}{2}.
\]
This implies $\inf \Omega \geq  \min \{0, \inf \Gamma -\frac{r}{2} \}.$ Therefore,          $
 \sup \Omega + \inf \Omega \geq  \frac{r}{2} >0. $
	Following Theorem \ref{continuous:characterization}, we get $f \not \perp_{\rho} g.$    This contradicts the fact that $f$ is $\rho$-right symmetric.
	  
	
	  \textbf{Case~2:}
	   Suppose that $\sup \Gamma + \inf \Gamma  < 0.$ Let $ \sup \Gamma + \inf \Gamma = -r.$ So, $r > 0.$ Clearly, $\inf \Gamma \leq -\frac{r}{2}.$
	 For any $v^* \in Ext(J(f(k_0))) ,$ let $G_{v^*}$ and $\mathcal{G}$ be the same set as described in Case 1. Similarly, let $U \subset K$ and $g \in C(K, \mathbb{X})$ be defined as in Case 1. Proceeding similarly as before we obtain $g \perp_{\rho} f.$ 
	  Let 
	 $$\Omega = \bigg\{ y^*(g(k)): k \in K, y^* \in Ext(B_{\mathbb{X}^*}), y^*(f(k))=1 \bigg\}.$$
	  Observe that  $\Gamma \subset \Omega $ and so, $\inf  \Omega \leq  \inf \Gamma < -\frac{r}{2} < 0.$    Suppose $\sup \Gamma = u^*(w),$ for some $u^* \in Ext(J(f(k_0))).$ 
	  Let $k \in U$ and $ y^* \in J(f(k)).$
	    Again proceeding similarly, there exists $v^* \in Ext(J(f(k_0)))$ such that $|y^*(w)- v^*(w)| < \frac{r}{2}.$
	    So, we have 
	    \[
	  v^*(w) - \frac{r}{2} < y^*(w)<  v^*(w) + \frac{r}{2} \leq u^*(w) + \frac{r}{2}.
	    \]
	    When $y^*(w) \leq 0,$ we have  $\alpha(k)y^*(w) \leq 0$ and  
	    whenever $y^*(w) > 0,$ we have  $\alpha(k)y^*(w) \leq  y^*(w). $
	    Therefore, 
	    \[
	    \alpha(k)y^*(w) \leq  y^*(w) < 	x_0^* (w) + \frac{r}{2}.
	    \]
	    This implies $\sup \Omega \leq \max \{0, \sup \Gamma +\frac{r}{2} \}.$ 
	    
	    Therefore,          $
	    	\sup \Omega + \inf \Omega \leq  -\frac{r}{2} <0. $
	  Following Theorem \ref{continuous:characterization}, we get $f \not \perp_{\rho} g.$      
	This contradicts the fact that $f$ is $\rho$- right symmetric. This completes the  theorem.

	              
	\end{proof}

		\begin{theorem}\label{left:rho:C(K)}
		Let $K$ be a  compact, perfectly normal space.  Then   $f \in S_{C(K)}$ is $\rho$-left symmetric if and only if   $|f(k)|< 1$ implies that	 $f(k) = 0.$
			\end{theorem}
	
		\begin{proof}
		We only prove the sufficient part as the necessary part follows immediately from Theorem \ref{left:rho:continuous}.  Let $f \perp_{\rho} g.$ Suppose 
		\[
		\Omega= \{ sgn(f(k)) g(k): k \in K, |f(k)|=\|f\|=1\}.
		\]
		From  Corollary \ref{C(K)}, we get $\sup \Omega+ \inf \Omega=0.$ 
		 Suppose 
		\[
		\Omega_1= \{ sgn(g(k)) f(k): k \in K, |g(k)|=\|g\|=1\}.
		\]
			Clearly, $\sup \Omega_1 \leq 1$ and $\inf \Omega_1 \geq -1.$
		First 	consider the case that $M_f$ is singleton. Let $M_f=\{k_0\},$ for some $k_0 \in K.$ Then $sgn(f(k_0))g(k_0)=0.$ This implies $g(k_0)=0.$ So, $M_g \subseteq K \setminus \{k_0\}.$ Now as $M_f=\{k_0\}$ and $|f(k)|< 1 \implies f(k)=0$, we easily get that for any $k \in M_g, f(k)=0.$ Therefore, $\Omega_1 = \{0\}.$ From Corollary \ref{C(K)}, we get $g \perp_{\rho} f.$
			
			Now consider $M_f$ is not singleton.
		  As $\Omega$ is compact, we assume $\sup \Omega= sgn(f(k_1)) g(k_1)$ and $\inf \Omega= sgn(f(k_2)) g(k_2),$ where $|f(k_1)|=|f(k_2)|=1.$
		 So 
		 \begin{eqnarray}\label{eqn:1}
		 	sgn(f(k_1)) g(k_1)+  sgn(f(k_2)) g (k_2)=0.
		 \end{eqnarray} 
		This implies $|g(k_1)|=|g(k_2)|.$ Now two cases can be possible:
		
		Case 1: Let $|g(k_1)|=|g(k_2)|<1.$ As $\sup \Omega= sgn(f(k_1)) g(k_1)$ and $\inf \Omega= sgn(f(k_2)) g(k_2),$ it is clear that $M_f \cap M_g = \emptyset.$ So, for any $k \in M_g,$ $f(k)=0.$ Therefore, $\Omega_1=\{0\}$ and from Corollary \ref{C(K)}, we get $g \perp_{\rho} f.$
		
		Case 2: Let $|g(k_1)|=|g(k_2)|=1.$ If $g(k_1), g(k_2)$ have the same sign, then from equation (\ref{eqn:1}), $f(k_1), f(k_2)$ have the different signs.  Since $|f(k_1)|= |f(k_2)|=1$, it is immediate that $	sgn(g(k_1)) f(k_1)+  sgn(g(k_2)) f (k_2)=0.$ Moreover, as $|sgn(g(k_1)) f(k_1)|=|sgn(g(k_2)) f(k_2)|=1 $ it is easy to observe that  $\sup \Omega_1 + \inf \Omega_1 =0.$  Similarly, for the case when $g(k_1), g(k_2)$ have the different sign, using analogous argument we can show $\sup \Omega_1+ \inf \Omega_1=0.$ Therefore, from Corollary \ref{C(K)}, we get $g \perp_{\rho} f.$  This completes the theorem.
	\end{proof}

		\begin{cor}
		Let $K$ be a  compact, connected, perfectly normal space.  Then   $f \in S_{C(K)}$ is $\rho$-left symmetric if and only if   $f(k)=1, \forall k \in K$ or $f(k)=-1, \forall k \in K.$
	\end{cor}
	
		\begin{theorem}\label{right:rho:C(K)}
		Let $K$ be a  compact, perfectly normal space.  Then   $f \in S_{C(K)}$ is $\rho$-right symmetric if and only if   $f$ satisfies the following : 
		\begin{itemize}
			\item[(i)] $f(k) \neq 0, ~ \forall k \in K$
			\item[(ii)] $|f(k_1)| = |f(k_2)|$ implies that $|f(k_1)|=|f(k_2)|=1$
		\end{itemize} 
	\end{theorem}

	\begin{proof}
	We only prove the sufficient part as the necessary part follows immediately from Theorem \ref{right:rho:continuous}. Let $g \perp_{\rho} f.$ Suppose 
	\[
	\Omega= \{ sgn(g(k)) f(k): k \in K, |g(k)|=\|g\|=1\}.
	\]
	From  Corollary \ref{C(K)}, we get $\sup \Omega+ \inf \Omega=0.$ As $\Omega$ is compact, we assume $\sup \Omega= sgn(g(k_1)) f(k_1)$ and $\inf \Omega= sgn(g(k_2)) f(k_2),$ where $|g(k_1)|=|g(k_2)|=1.$ So $$sgn(g(k_1)) f(k_1)+  sgn(g(k_2)) f(k_2)=0.$$ Therefore, $|f(k_1)|=|f(k_2)|$. Since $f$ satisfies the condition (ii), $|f(k_1)|=|f(k_2)|=1=\|f\|.$ 	Let
	\[
	\Omega_1= \{ sgn(f(k)) g(k): k \in K, |f(k)|=\|f\|=1\}.
	\]
	Clearly, $\sup \Omega_1 \leq 1$ and $\inf \Omega_1 \geq -1.$
	Let us now consider the cases: 
	
	Case 1: Let $f(k_1), f(k_2)$ have the same sign. Then $g(k_1), g(k_2)$ both have different signs. So, $sgn(f(k_1)) g(k_1)$, $sgn(f(k_2)) g(k_2)$ both have different sign and $|sgn(f(k_1)) g(k_1)|= |sgn(f(k_2)) g(k_2)|=1.$ It is now immediate that $\sup \Omega_1+\inf \Omega_1=0.$ Applying Corollary \ref{C(K)}, we get $f \perp_{\rho} g.$
	
	Case 2:  Let $f(k_1), f(k_2)$ have different signs. Then $g(k_1), g(k_2)$ both have the same sign. Again following arguments as in Case 1, we can easily show $f \perp_\rho g.$\\
	This completes the proof.

	\end{proof}
	
	As an immediate consequence of Theorems \ref{left:continuous},  \ref{right:continuous}, \ref{left2:continuous}, and \ref{right2:continuous}  we obtain the following corollary which characterize $(\rho_{+})$-left (right) symmetric and $(\rho_{-})$-left (right) symmetric points in $\ell_{\infty}^n(\mathbb{X}).$
	
	
	\begin{cor}\label{directsum}
		Let $\mathbb{X}$ be a Banach space and $\widetilde{x}=(x_1, x_2, \ldots, x_n) \in S_{\ell_{\infty}^n(\mathbb{X})}.$
		\begin{itemize}
			\item[(i)] $\widetilde{x}$ is $(\rho_{+})$-left symmetric if and only if there exists $i \in \{1, 2, \ldots, n\}$ such that $\|x_i\|=1, x_j=0, \forall j \neq i$ and $x_i$ is a $(\rho_{+})$-left symmetric points of $\mathbb{X}.$
			
			\item[(ii)] $\widetilde{x}$ is $(\rho_{+})$-right symmetric if and only if      exactly one of the following holds: 
			
			\begin{itemize}
			
			\item[(a)] If $\|x_i\|=1, \|x_j\|< 1, \forall j\neq i$ then $x_i$ is a $(\rho_{+})$-right symmetric point in $\mathbb{X}$ and for any $k \in \{1, 2, \ldots, n\} \setminus \{i\},$  there exists no $w \in S_{\mathbb{X}}$ such that $w \perp_{\rho_+} x_k$
			
			\item[(b)] If $\|x_i\|=\|x_j\|=1, $ for some $i, j \in \{1, 2, \ldots, n\}$ then for any $k \in \{1, 2, \ldots, n\}$ there exists no $w \in S_{\mathbb{X}}$ such that $w \perp_{\rho_+} x_k.$
			\end{itemize}
			
				\item[(iii)] $\widetilde{x}$ is $(\rho_{-})$-left symmetric if and only if there exists $i \in \{1, 2, \ldots, n\}$ such that $\|x_i\|=1, x_j=0, \forall j \neq i$ and $x_i$ is a $(\rho_{-})$-left symmetric points of $\mathbb{X}.$
			
			\item[(iv)] $\widetilde{x}$ is $(\rho_{-})$-right symmetric if and only if     exactly one of following holds: 
			
			\begin{itemize}
					\item [(a)] If $\|x_i\|=1, \|x_j\|< 1, \forall j\neq i$ then $x_i$ is a $(\rho_{-})$-right symmetric point in $\mathbb{X}$ and for any $k \in \{1, 2, \ldots, n\} \setminus \{i\},$  there exists no $w \in S_{\mathbb{X}}$ such that $w \perp_{\rho_-} x_k$
				
				\item [(b)] If $\|x_i\|=\|x_j\|=1, $ for some $i, j \in \{1, 2, \ldots, n\}$ then for any $k \in \{1, 2, \ldots, n\}$ there exists no $w \in S_{\mathbb{X}}$ such that $w \perp_{\rho_-} x_k.$
				\end{itemize}

		\end{itemize}
	\end{cor}

	 The next corollary provides necessary condition for $\rho$-left (right) symmetric points of $\ell_{\infty}^n(\mathbb{X}).$
	 
	 \begin{cor}\label{directsum2}
	 		Let $\mathbb{X}$ be a Banach space and $\widetilde{x}=(x_1, x_2, \ldots, x_n) \in S_{\ell_{\infty}^n(\mathbb{X})}.$
	 		\begin{itemize}
	 				\item [(i)] If $\widetilde{x}$ is $\rho$-left symmetric then $\|x_k\|< 1$ implies that $x_k=0.$   
	 			
	 			\item[(ii)] If $\widetilde{x}$ is $\rho$-right symmetric then the following holds:
	 			
	 			\subitem(a) $x_k \neq 0, \forall k \in \{1, 2, \ldots, n\}$
	 			\subitem(b) $\|x_i\|=\|x_j\|$ implies that $\|x_i\|=\|x_j\|=1$
	 			\subitem(c) $\|x_i\|=1$ implies that $x_i$ is $\rho$-right symmetric.
	 		\end{itemize}
	 \end{cor}
		
		\section*{\textbf{Symmetricity of $\rho$-orthogonality in the space of operators}}
		
		In this section we study the symmetric elements with respect to $\rho_{+}$-orthogonality, $\rho_{-}$-orthogonality and $\rho$-orthogonality in the space of bounded linear operators.
		
		\begin{prop}
				Let $\mathbb{X}$ be a  reflexive Banach space, $ \mathbb{Y}$ be any Banach space and  let $T , S \in \mathbb{K}(\mathbb{X}, \mathbb{Y}).$ Suppose that 
				\[
				\Omega = \bigg\{ y^*(Sx): x \in Ext(B_{\mathbb{X}}), y^* \in Ext(B_{\mathbb{Y}^*}), y^*(Tx)=\|T\| \bigg\}.
				\]
				Then 
					\begin{itemize}
					\item [(i)] $ T \perp_{\rho_+}  S$ if and only if $\sup \Omega=0.$
					\item [(ii) ] $ T \perp_{\rho_-}  S$ if and only if $\inf \Omega=0.$
					\item [(iii)] $ T \perp_{\rho}  S$ if and only if $\sup \Omega + \inf \Omega=0.$
				\end{itemize}  
		\end{prop}

		\begin{proof}
			As $\mathbb{X}$ is reflexive, from \cite{LO}, we get
			\[
			Ext(B_{\mathbb{K}(\mathbb{X}, \mathbb{Y})^*})=\{ y^* \otimes x: x \in Ext (B_{\mathbb{X}}), y^* \in Ext(B_{\mathbb{Y}^*})\},
			\]
			where $y^* \otimes x \in \mathbb{K}(\mathbb{X}, \mathbb{Y})^*$ and is defined as $ [y^* \otimes x] (T)= y^*(Tx),$ for any $T \in \mathbb{K}(\mathbb{X}, \mathbb{Y}).$
			Observe that 
		\begin{eqnarray*}
			\rho_{+}'(T, S) &=& \sup \{ \psi(S): \psi \in Ext(J(T))\}\\
			 &=& \sup \{ \psi(S): \psi \in Ext(B_{\mathbb{K}(\mathbb{X}, \mathbb{Y})^*}), \psi(T)= \|T\|\}\\
			&= & \sup \{ [y^* \otimes x] (S): x \in Ext(B_{\mathbb{X}}), y^* \in Ext(B_{\mathbb{Y}^*}), [y^* \otimes x] (T)=\|T\| \}\\
				&= & \sup \{ y^*(Sx): x \in Ext(B_{\mathbb{X}}), y^* \in Ext(B_{\mathbb{Y}^*}), y^*(Tx)=\|T\| \}.
		\end{eqnarray*}
		
		Similarly, 
			\begin{eqnarray*}
			\rho_{-}'(T, S)= \inf \{y^*(Sx): x \in Ext(B_{\mathbb{X}}), y^* \in Ext(B_{\mathbb{Y}^*}),  y^*(Tx)=\|T\| \}.
		\end{eqnarray*}
		Hence the result.
		\end{proof}

		\begin{theorem}\label{left:rank1}
			Let $T \in \mathbb{K}(\mathbb{X}, \mathbb{Y})$ be  a $rank~1$ operator, where  $\mathbb{X}, \mathbb{Y}$ are   reflexive Banach spaces. Suppose $T(x)= f(x)y, ~\forall x \in \mathbb{X},$ where $y \in \mathbb{Y}, f \in \mathbb{X}^*.$
			\begin{itemize}
				\item [(i)] If $T$ is $(\rho_{+})$-left symmetric then  $y$ and $f$ are  $(\rho_{+})$-left symmetric points of $\mathbb{Y}$ and $\mathbb{X}^*,$ respectively.

					\item [(ii)]  If $T$ is $(\rho_{+})$-right symmetric then  $y$ and $f$ are  $(\rho_{+})$-right symmetric points of $\mathbb{Y}$ and $\mathbb{X}^*,$ respectively.
				
					\item [(iii)] If $T$ is $(\rho_{-})$-left symmetric then  $y$ and $f$ are $(\rho_{-})$-left symmetric points of $\mathbb{Y}$ and $\mathbb{X}^*,$ respectively.

				\item [(iv)]  If $T$ is $(\rho_{-})$-right symmetric then  $y$ and $f$ are  $(\rho_{-})$-right symmetric points of $\mathbb{Y}$ and $\mathbb{X}^*,$ respectively.
				
					\item[(v)] If $T$ is $\rho$-left symmetric then  $y$ and $f$ are  $\rho$-left symmetric points of $\mathbb{Y}$ and $\mathbb{X}^*,$ respectively.
				
				\item[(vi)] If $T$ is $\rho$-right symmetric then  $y$ and $f$ are $\rho$-right symmetric points of $\mathbb{Y}$ and $\mathbb{X}^*,$ respectively.
			\end{itemize}
		\end{theorem}

		\begin{proof}
		 (i)   Suppose on the contrary that $y$ is not $(\rho_+)$-left symmetric. Then there exists $v \in \mathbb{Y}$ such that $y \perp_{\rho_+} v$ but $v \not\perp_{\rho_{+}} y.$ Define $S : \mathbb{X} \to \mathbb{Y}$ such that $S(x)= f(x) v,$ for any $x \in \mathbb{X}.$ It is straightforward to verify that $M_T= M_S$ and for any $z \in M_T(= M_S),$ $Tz= \pm y, Sz= \pm v.$ Let 
		 \[
		 	\Omega = \bigg\{ y^*(Sx): x \in  Ext(B_{\mathbb{X}}), y^* \in Ext(B_{\mathbb{Y}^*}), y^*(Tx)=\|T\| \bigg\}.
		 \]
		 It is easy to observe that whenever $y^*(Tx)=\|T\|,$ then $x \in M_T$ and $Sx= v.$ So, 
		 \[
		 	\Omega = \{ y^*(v):  y^* \in Ext( B_{\mathbb{Y}^*}), y^*(y)=\|y\| \}= \{ y^*(v): y^* \in Ext (J(y)) \}.
		 \]
		Therefore, $\rho_{+}'(T, S)= \rho_{+}'(y,v)$. 
		Since $y \perp_{\rho_+} v$ and $v \not\perp_{\rho_{+}} y,$ it is immediate that $T \perp_{\rho_{+}} S$ and $S \not\perp_{\rho_{+}} T.$ This contradicts that $T$ is $(\rho_{+})$-left symmetric. So $y$ is $(\rho_{+})$-left symmetric. 
		It is easy to check that $T^*(y^*)=  \psi(y)(y^*) f,$ for any $y^* \in \mathbb{Y}^*,$ where $\psi$ is the canonical isometric isomorphism from $\mathbb{X}$ to $\mathbb{X}^{**}.$ Clearly, $T^*$ is rank $1$. Moreover, as $\mathbb{X}, \mathbb{Y}$ are reflexive,  $T$ is $(\rho_{+})$-left symmetric implies that $T^*$ is also $(\rho_{+})$-left symmetric. Now, suppose on the contrary that $f$ is not $(\rho_{+})$-left symmetric. Then there exists $g \in \mathbb{X}^*$ such that $f \perp_{\rho_{+}} g$ but $g \not \perp_{\rho_{+}} f.$ Define $A: \mathbb{Y}^* \to \mathbb{X}^*$ such that $A(y^*)= \psi(y)(y^*) g.$ Clearly, $M_{T^*}= M_A.$ Proceeding similarly as above, we obtain that $T^* \perp_{\rho_{+}} A$ but $A \not\perp_{\rho_{+}} T^*,$ which contradicts that $T^*$ is left symmetric. This completes the proof of (i). 
		
	Proofs of	(ii), (iii), (iv), (v), (vi) follows similarly. 
		\end{proof}

		\begin{prop}\label{prop:isometric}
			\begin{itemize}
				\item[(i)] The space  $\mathbb{L}(\ell_{1}^n, \mathbb{X})$ is isometrically isomorphic to $\ell_{\infty}^n(\mathbb{X}),$ for any Banach space $\mathbb{X}.$ 
				
				\item[(ii)]  The space  $\mathbb{K}( \mathbb{X}, C(K))$ is isometrically isomorphic to $C(K, \mathbb{X}^*),$ where $K$ is a compact Hausdorff space.
				
				\item[(iii)]  The space $\mathbb{L}( \mathbb{X}, \ell_{\infty}^n)$ is isometrically isomorphic to $\ell_{\infty}^n(\mathbb{X}^*),$
				for any Banach space $\mathbb{X}.$ 			
			\end{itemize}
		\end{prop}
		
		For a Banach space $\mathbb{X}$ with $Ext(B_{\mathbb{X}}) \neq \emptyset,$  let $e(B_{\mathbb{X}})= \{ x, y \in Ext(B_{\mathbb{X}}): x \neq \pm y\}.$ 
		
	     \begin{theorem}\label{nice:left}
				Let  $T \in S_{\mathbb{L}(\mathbb{X}, \mathbb{Y})}.$
				 Then the following hold: 
				 
				 \begin{itemize}
				 	\item[(i)]  $T$ is $(\rho_{+})$-left symmetric if  there exists $y_0^* \in Ext(B_{\mathbb{Y}^*})$ such that $T^*(y_0^*) $ is a $(\rho_+)$-left symmetric point and $T^*(y^*)=0, \forall y^* \in Ext(B_{\mathbb{Y}^*}) \setminus \{\pm y_0^*\}.$
				 	
				 	\item[(ii)]  $T$ is $(\rho_{-})$-left symmetric if  there exists $y_0^* \in Ext(B_{\mathbb{Y}^*})$ such that $T^*(y_0^*) $ is a $(\rho_-)$-left symmetric point and $T^*(y^*)=0, \forall y^* \in Ext(B_{\mathbb{Y}^*}) \setminus \{\pm y_0^*\}.$
				 \end{itemize}
		\end{theorem}
		
		\begin{proof}
(i)
			We define $ \phi: \mathbb{L}(\mathbb{X},\mathbb{Y})  \longrightarrow   \ell_{\infty}(e(B_{\mathbb{Y}^*}), \mathbb{X}^*) $  as  
			$$ 	\phi(T)\quad  =  \quad f_T \quad   \forall  T \in \mathbb{L}(\mathbb{X}, \mathbb{Y}) , $$
			where $ f_T:  e(B_{\mathbb{Y}^*}) \to \mathbb{X}^* $ is given by 
			$$ f_T( y^*)=  T^*( y^*), \, \, \forall y^* \in e(B_{\mathbb{Y}^*}).$$
		It is easy to show that $\phi$ is a linear isometry. So, $\mathbb{L}(\mathbb{X}, \mathbb{Y})$ is embedded into $\ell_{\infty}(e(B_{\mathbb{Y}^*}), \mathbb{X}^*).$
			As 	there exists $y_0^* \in e(B_{\mathbb{Y}^*})$ such that $f_T(y_0^*) $ is a $(\rho_{+})$-left symmetric point and $f_T(y^*)=0, \forall y^* \in e(B_{\mathbb{Y}^*}) \setminus \{ y_0^*\}, $ it follows from Theorem \ref{left:bounded} that $ f_T$ is left symmetric in $\ell_{\infty}(e(B_{\mathbb{Y}^*}), \mathbb{X}^*).$ Clearly, $f_T$ is also  left symmetric  in $\phi(\mathbb{L}(\mathbb{X}, \mathbb{Y})).$ Since  left symmetricity is preserved under isometric isomorphism, $T$ 
			is left symmetric in $\mathbb{L}(\mathbb{X}, \mathbb{Y}).$
			
			(ii) follows similarly.	
		\end{proof}

		\begin{remark}\label{prop:embedd}
	Observe that	  $\mathbb{K}(\mathbb{X}, \mathbb{Y})$ is  embedded into $C(e(B_{\mathbb{Y}^*}), \mathbb{X}^*).$ The isometric embedding $\phi : \mathbb{K}(\mathbb{X}, \mathbb{Y}) \to C(e(B_{\mathbb{Y}^*}), \mathbb{X}^*)$ is defined as $\phi(T), = f_T$ where $f_T (y^*)= T^*(y^*), ~ \forall y^* \in (e(B_{\mathbb{Y}^*})).$
		\end{remark}

		We next present a sufficient condition for $(\rho_{+})$-right symmetric and $(\rho_-)$-right symmetric in $\mathbb{K}(\mathbb{X}, \mathbb{Y}).$ The proof is omitted as it is in the same line as Theorem \ref{nice:left}, considering $\mathbb{K}(\mathbb{X}, \mathbb{Y})$ as a subspace $C(e(B_{\mathbb{Y}^*}), \mathbb{X}^*).$
		
		    \begin{prop}\label{nice:right}
			Let $T \in S_{\mathbb{K}(\mathbb{X}, \mathbb{Y})}.$ Then the following hold: 
			
			\begin{itemize}
				\item[(i)] Suppose  $T$ satisfies one of the following:
			\begin{itemize}
					\item[(a)] If $M_{T^*}=\{y_0^*\}$ then $T^*(y_0^*)$ is $(\rho_{+})$-right symmetric and for any $y^* \in Ext(B_{\mathbb{X}^*}) \setminus \{\pm y_0^*\},$  $x^* \perp_{\rho_+} T^*(y^*) \implies x^*=0.$
				
				\item[(b)] If $M_{T^*}$ is not singleton and for any $y^* \in M_{T^*},$  $x^* \perp_{\rho_+} T^*(y^*) \implies x^*=0.$
			\end{itemize}
				Then 	 $T$ is $(\rho_{+})$-right symmetric.\\
				
				\item[(ii)]  Suppose  $T$ satisfies one of the following:
				\begin{itemize}
						\item [(a)] If $M_{T^*}=\{y_0^*\}$ then $T^*(y_0^*)$ is $(\rho_{-})$-right symmetric and for any $y^* \in Ext(B_{\mathbb{X}^*}) \setminus \{\pm y_0^*\},$  $x^* \perp_{\rho_-} T^*(y^*) \implies x^*=0.$
					
					\item [(b)] If $M_{T^*}$ is not singleton and for any $y^* \in M_{T^*},$  $x^* \perp_{\rho_-} T^*(y^*) \implies x^*=0.$
				\end{itemize}
			Then 	 $T$ is $(\rho_{-})$-right symmetric.
			\end{itemize}
	
		\end{prop}

		As a consequence of Corollary \ref{directsum}  we obtain the $(\rho_{+})$-left (right) symmetric and $(\rho_{-})$-left (right) symmetric operators in $\mathbb{L}(\ell_{1}^n, \mathbb{X})$.
		
		\begin{cor}\label{cor1}
			Let $T \in S_{\mathbb{L}(\ell_{1}^n , \mathbb{X})}.$ Then  
			\begin{itemize}
				\item[(i)] $T$  is a $(\rho_{+})$-left symmetric operator if and only if there exists $i \in \{1, 2, \ldots, n\}$ such that $\|T(e_i)\|=1, T(e_j)=0, \forall j \neq i$ and $T(e_i)$ is a $(\rho_{+})$-left symmetric point of $\mathbb{X}.$
				
				\item[(ii)] $T$  is a $(\rho_{-})$-left symmetric operator if and only if there exists $i \in \{1, 2, \ldots, n\}$ such that $\|T(e_i)\|=1, T(e_j)=0, \forall j \neq i$ and $T(e_i)$ is a $(\rho_{-})$-left symmetric point of $\mathbb{X}.$
			\end{itemize}
		\end{cor}

		\begin{proof}
			We only prove (i) as (ii) follows similarly.	From Proposition \ref{prop:isometric}(i), $\mathbb{L}(\ell_{1}^n, \mathbb{X})$ is isometrically isomorphic to $\ell_{\infty}^n(\mathbb{X}).$ Moreover, $\phi: \mathbb{L}(\ell_{1}^n, \mathbb{X}) \to \ell_{\infty}^n(\mathbb{X})$ defined as  $\phi(T)= (Te_1, Te_2, \ldots, Te_n)$ is the isometric isomorphism. Since  $(\rho_{\pm})$-left symmetricity is preserved under isometric isomorphism, the desired result follows easily from Corollary \ref{directsum}.
		\end{proof}
		
		The next proof is omitted as it is in the same line of Corollary \ref{cor1}.
		
		\begin{cor}
			Let $T \in S_{\mathbb{L}(\ell_{1}^n , \mathbb{X})}.$ Then 
			 
			\begin{itemize}
				\item[(i)] $T$	 is a $(\rho_{+})$-right symmetric if and only if and $T$ satisfies exactly one of following:
				\begin{itemize}
						\item [(a)] If $M_T= \{ \pm e_i\}, $ for some $ 1\leq i \leq n,$ then $T(e_i)$ is a $(\rho_{+})$-right symmetric point in $\mathbb{X}$ and  for any $k \in \{1, 2, \ldots, n\} \setminus \{i\},$   $w \perp_{\rho_+} T(e_k) \implies w = 0.$
					
					\item [(b)] If $e_i, e_j \in M_T$ for some $i, j \in \{1, 2, \ldots, n\}$ then for any $k \in \{1, 2, \ldots, n\},$  $w \perp_{\rho_+} T(e_k) \implies w=0.$
				\end{itemize}

				\item[(ii)] $T$	 is a $(\rho_{-})$-right symmetric if and only if and $T$ satisfies exactly one of following:
				\begin{itemize}
						\item[(a)] If $M_T= \{ \pm e_i\}, $ for some $ 1\leq i \leq n,$ then $T(e_i)$ is a $(\rho_{-})$-right symmetric point in $\mathbb{X}$ and  for any $k \in \{1, 2, \ldots, n\} \setminus \{i\},$   $w \perp_{\rho_-} T(e_k) \implies w=0.$
					
					\item [(b)] If $e_i, e_j \in M_T$ for some $i, j \in \{1, 2, \ldots, n\}$ then for any $k \in \{1, 2, \ldots, n\}$ $w \perp_{\rho_-} T(e_k) \implies w=0.$
				\end{itemize}
			
			\end{itemize}
				\end{cor}

		Next we present  necessary conditions for  $\rho$-left (right) symmetric operators of $\mathbb{L}(\ell_{1}^n , \mathbb{X})$ as a consequence of Corollary \ref{directsum2}. The proofs are omitted as it follows similarly as Corollary \ref{cor1}.
		
			\begin{cor}
			Let $\mathbb{X}$ be a Banach space. Suppose $T \in S_{\mathbb{L}(\ell_{1}^n , \mathbb{X})}$	 is a $\rho$-left symmetric. Then $T(e_i)=0,$ for any $e_i \notin M_T.$ 
				
			
		\end{cor}

		\begin{cor}
				Let $\mathbb{X}$ be a Banach space. Suppose $T \in S_{\mathbb{L}(\ell_{1}^n , \mathbb{X})}$	 is a $\rho$-right  symmetric. Then the following holds: 
				\begin{itemize}
						\item[(i)] $T(e_k) \neq 0, ~\forall k \in \{1, 2, \ldots, n\}$
					\item[(ii)] $\|T(e_i)\|=\|T(e_j)\|$ implies that $\|T(e_i)\|=\|T(e_j)\|=1$
					\item[(iii)] $\|T(e_i)\|=1$ implies that $T(e_i)$ is $\rho$-right symmetric.\\
				\end{itemize}
		
		\end{cor}

We next study $(\rho_{+})$-left (right) symmetric points, $(\rho_{-})$-left (right) symmetric points and $ \rho$-symmetric points in the space $\mathbb{K}( \mathbb{X}, C(K))$. 
		
		\begin{cor}\label{left1}
			Let $K$ be a compact, perfectly normal space and  let  $T \in S_{\mathbb{K}(\mathbb{X}, C(K))}.$ Then   
			\begin{itemize}
				\item[(i)] $T$ is ($\rho_{+}$)-left symmetric if and only if there exists exactly one point $k_0 \in K$ such that $T^*(\delta_{k_0}) \in S_{\mathbb{X}}$ is a $(\rho_{+})$-left symmetric point in $\mathbb{X}^*$ and $T^*(\delta_{k})=0,$  $\forall k \in K \setminus \{k_0\}.$
				
				\item[(ii)] $T$ is ($\rho_{-}$)-left symmetric if and only if there exists exactly one point $k_0 \in K$ such that $T^*(\delta_{k_0}) \in S_{\mathbb{X}}$ is a $(\rho_{-})$-left symmetric point in $\mathbb{X}^*$ and $T^*(\delta_{k})=0,$  $\forall k \in K \setminus \{k_0\},$
			\end{itemize} 
			where $\delta_k \in (C(K))^*$ such that  $\delta_k(f)=f(k),$ for any $f \in C(K).$
			Moreover, if $K$ does not contain any isolated point then $T$ is $(\rho_{\pm})$-left symmetric if and only $T$ is the zero operator. 
		\end{cor}
		
	\begin{proof}
		We only prove (i) as (ii) follows similarly.	From Proposition \ref{prop:isometric}(ii), $\mathbb{K}( \mathbb{X}, C(K))$ is isometrically isomorphic to $ C(K, \mathbb{X}^*).$ Note that $\phi: \mathbb{K}( \mathbb{X}, C(K)) \to C(K, \mathbb{X}^*)$ given by  $\phi(T)(k)= T^*(\delta_{k})$ is the isometric isomorphism ( see \cite{DS}).  Hence applying Theorem \ref{left:continuous}, the result follows easily. 
	\end{proof}

			\begin{cor}
				Let $K$ be a compact, perfectly normal space and let $T \in S_{\mathbb{K}(\mathbb{X}, C(K))}.$ Then  
			\begin{itemize}
				
				\item [(i)]    $T$ is ($\rho_{+}$)-right symmetric if and only if  $T$ satisfies  exactly one of following: 
				\begin{itemize}
						\item[(a)] If $M_{T^*}= \{ \pm \delta_{k_0} \}, $ for some $ k_0 \in K$ then $T^*(\delta_{k_0})$ is a $(\rho_{+})$-right symmetric point in $\mathbb{X}^*$ and  for any $k \in \{1, 2, \ldots, n\} \setminus \{k_0\},$  $w \perp_{\rho_+} T^*(\delta_k) \implies w=0.$
					
					\item [(b)] If $\delta_{i}, \delta_{j} \in M_{T^*}$ for some $i, j \in K$ then for any $k \in K$ $w \perp_{\rho_\pm} T^*(\delta_k) \implies w=0.$
				\end{itemize}

					\item [(ii)]    $T$ is ($\rho_{-}$)-right symmetric if and only if  $T$ satisfies  exactly one of following: 
					\begin{itemize}
						\item[(a)] If $M_{T^*}= \{ \pm \delta_{k_0} \}, $ for some $ k_0 \in K$ then $T^*(\delta_{k_0})$ is a $(\rho_{-})$-right symmetric point in $\mathbb{X}^*$ and  for any $k \in \{1, 2, \ldots, n\} \setminus \{k_0\},$  $w \perp_{\rho_-} T^*(\delta_k) \implies w=0.$
					
					\item [(b)] If $\delta_{i}, \delta_{j} \in M_{T^*}$ for some $i, j \in K$ then for any $k \in K$ $w \perp_{\rho_\pm} T^*(\delta_k) \implies w=0.$
					\end{itemize}
			
			\end{itemize}
		\end{cor}
		
		\begin{proof}
			The proof is in the same line as that of Proposition \ref{left1}, and is therefore omitted. 
		\end{proof}

		\begin{cor}
				Let $K$ be a compact, perfectly normal space. Let  $T \in S_{\mathbb{K}(\mathbb{X}, C(K))}$  be $\rho$-left symmetric. Then   $T^*(\delta_{k})=0$, for any $\delta_k \notin M_{T^*}.$
			\end{cor}

			\begin{cor}
			Let $K$ be a compact, perfectly normal space and let  $T \in S_{\mathbb{K}(\mathbb{X}, C(K))}.$ If  $T$ is $\rho$-right symmetric then the following holds:
			\begin{itemize}
				\item[(i)] $T^*(\delta_{k}) \neq 0, \forall k \in K.$
				\item[(ii)] $\|T^*(\delta_{i})\|=\|T^*(\delta_{j})\|$ implies that $\|T^*(\delta_{i})\|=\|T^*(\delta_{j})\|=1.$
				\item[(iii)] $\|T^*(\delta_{i})\|=1$ implies that $T^*(\delta_{i})$ is $\rho$-right symmetric.\\
		\end{itemize}
		\end{cor}

Finally we  study  $(\rho_{+})$-left (right) symmetric points, $(\rho_{-})$-left (right) symmetric points and $ \rho$-symmetric points in the space 
		$\mathbb{L}( \mathbb{X}, \ell_{\infty}^n)$.
		\begin{cor}
			Let  $T \in S_{\mathbb{L}( \mathbb{X}, \ell_{\infty}^n)}. $ Then 
			\begin{itemize}
				\item[(i)] $T$ is  $(\rho_{+})$-left symmetric  if and only  there exists $j \in \{1, 2, \ldots, n\}$  such that for any $x \in \mathbb{X},$ $Tx=(0, \ldots, 0, \underset{j\text{-th}}{f(x)}, 0, \ldots, 0),$  where $f \in S_{\mathbb{X}^*}$ is a  $(\rho_+)$-left symmetric point of $ \mathbb{X}^*$.
				
				\item[(ii)] $T$ is  $(\rho_{-})$-left symmetric  if and only  there exists $j \in \{1, 2, \ldots, n\}$  such that for any $x \in \mathbb{X},$ $Tx=(0, \ldots, 0, \underset{j\text{-th}}{f(x)}, 0, \ldots, 0),$  where $f \in S_{\mathbb{X}^*}$ is a  $(\rho_-)$-left symmetric point of $ \mathbb{X}^*$.
			\end{itemize}
		\end{cor}
		
		\begin{proof}
		We only prove (i) as (ii) follows similarly.	Observe that $Ext(B_{(\ell_{\infty}^n)^*})= \{\pm e_i:  1 \leq i \leq n\},$ where $e_i(x)=x_i,$ for any $x=(x_1, x_2, \ldots, x_n)\in \ell_{\infty}^n.$	From Proposition \ref{prop:isometric}, $\mathbb{L}(\mathbb{X}, \ell_{\infty}^n)$ is isometrically isomorphic to $\ell_{\infty}^n(\mathbb{X}^*).$ Moreover, if $\phi$ is the isometric isomorphism between $\mathbb{L}(\mathbb{X}, \ell_{\infty}^n)$ and    $\ell_{\infty}^n(\mathbb{X}^*),$ then $\phi(T)=(T^*e_1, T^* e_2, \ldots, T^* e_n).$ 
			
			From Corollary \ref{directsum}, we obtain that $\phi(T)$ is $(\rho_{+})$-left symmetric if and only if there exists $j \in \{1, 2, \ldots, n\} $ such that $T^*(e_{j}) \in S_{\mathbb{X}^*}$ is a $(\rho_{+})$-left symmetric point and $T^*(e_i)=0,$ for any $i\in \{1,2,\ldots,n\} \setminus \{ j\}.$  Suppose that  $T^*(e_{j})=f \in S_{\mathbb{X}^*}.$
			Let for any $x \in \mathbb{X},$ $Tx= (u_1(x), u_2(x), \ldots, u_n(x)),$ where $u_i \in \mathbb{X}^*.$ Then $u_{j}(x)=e_{j}(Tx)= T^*(e_{j})(x)=f(x)$ and $u_i(x)=e_i(Tx)= T^*(e_i)(x)= 0,$ whenever $i \neq j.$ Thus we obtain (i).
		\end{proof}

		\begin{cor}
			Let  $T \in S_{\mathbb{L}(\mathbb{X}, \ell_{\infty}^n)}.$ Then 
			\begin{itemize}
				\item[(i)] $T$  is a $(\rho_{+})$-right symmetric if and only if
				for any $x \in \mathbb{X},$ $Tx =(f_1(x), f_2(x), \ldots, f_n(x)),$ where $f_i \in \mathbb{X}^*$ such that 
				$f_i$ satisfies    exactly one of following: 
				\begin{itemize}
						\item [(a)] If $\|f_{i_0}\|=1, \|f_j\|<1, \forall j \neq i_0 $  then $f_{i_0}$ is a $(\rho_{+})$-right symmetric element in $\mathbb{X}^*$ and  $g \perp_{\rho_+} f_k \implies g=0, ~  \forall k \in \{1, 2, \ldots, n\} \setminus \{i_0\},$
					
					\item[(b)] If $\|f_i\| = \|f_j\| =1 $ for some $i, j \in \{1, 2, \ldots, n\}$ then for any $k \in \{1, 2, \ldots, n\}$  $ g \perp_{\rho_+} f_k \implies g=0.$
				\end{itemize}

					\item[(ii)] $T$  is a $(\rho_{-})$-right symmetric if and only if
				for any $x \in \mathbb{X},$ $Tx =(f_1(x), f_2(x), \ldots, f_n(x)),$ where $f_i \in \mathbb{X}^*$ such that 
				$f_i$ satisfies    exactly one of following:
				\begin{itemize}
						\item [(a)] If $\|f_{i_0}\|=1, \|f_j\|<1, \forall j \neq i_0 $  then $f_{i_0}$ is a $(\rho_{\pm})$-right symmetric element in $\mathbb{X}^*$ and $g \perp_{\rho_-} f_k \implies g=0, ~  \forall k \in \{1, 2, \ldots, n\} \setminus \{i_0\},$
					
					\item[(b)] If $\|f_i\| = \|f_j\| =1 $ for some $i, j \in \{1, 2, \ldots, n\}$ then for any $k \in \{1, 2, \ldots, n\}$  $ g \perp_{\rho_\pm} f_k \implies g=0.$
				\end{itemize} 
			
			\end{itemize}
				\end{cor}
		
		\begin{cor}
			Let $\mathbb{X}$ be a Banach space. Suppose $T \in S_{\mathbb{L}( \mathbb{X}, \ell_{\infty}^n)}$	 is a $\rho$-left symmetric. Then for any $x \in \mathbb{X},$ $Tx =(f_1(x), f_2(x), \ldots, f_n(x)),$ where $f_i \in \mathbb{X}^*$ and  $\|f_i\| <  1 $ implies that $f_i=0$
				
			
		\end{cor}

		\begin{cor}
			Let $\mathbb{X}$ be a Banach space. Suppose  $T \in S_{\mathbb{L}( \mathbb{X}, \ell_{\infty}^n)}$	 is a $\rho$-right symmetric. Then for any $x \in \mathbb{X},$ $Tx =(f_1(x), f_2(x), \ldots, f_n(x)),$ where $f_i \in \mathbb{X}^*$  satisfies the        following:
			\begin{itemize}
				\item[(i)] $f_k \neq 0, \forall k \in \{1, 2, \ldots, n\}$
				\item[(ii)] $\| f_i\|=\|f_j\|$ implies that $\|f_i\|=\|f_j\|=1$
				\item[(iii)] $\|f_i\|=1$ implies that $f_i$ is $\rho$-right symmetric element of $\mathbb{X}^*$.
			\end{itemize}
			
		\end{cor}
	\noindent 	\textbf{Conflict of interest}\\
		The authors have no relevant financial or non-financial interests to disclose.
		The authors have no competing interests to declare that are relevant to the content of this article.

	\end{document}